\documentclass[11pt,a4paper,twoside]{article}
\usepackage[UKenglish]{babel}
\usepackage[utf8x]{inputenc}
\usepackage[T1]{fontenc}
\usepackage[shortcuts]{extdash}
\usepackage{latexsym,amsfonts,amsmath,amsthm,amssymb}
\usepackage{fullpage}
\usepackage{xcolor,bm,bbm}
\usepackage{array}
\usepackage{paralist}
\usepackage{mathrsfs}

\newcommand*{\R}[0]{\mathbb{R}}
\newcommand*{\N}[0]{\mathbb{N}}
\newcommand*{\dif}[0]{\mathrm{d}} 
\newcommand*{\bsigma}[0]{\boldsymbol{\sigma}}
\newcommand*{\vek}[1]{\mathbf{#1}} 
\newcommand*{\B}[0]{\mathbb{B}}
\newcommand*{\D}[0]{\mathbb{D}}
\newcommand*{\eps}[0]{\varepsilon}

\newcommand*{\ee}[0]{\mathsf{e}}
\newcommand*{\limd}[1]{\xrightarrow[#1\to\infty]{\mathrm{d}}}
\newcommand*{\Ell}[0]{\mathcal{E}} 
\newcommand*{\nEll}[0]{\widetilde{\mathcal{E}}} 
\newcommand*{\inv}[1]{#1^{-1}}
\newcommand*{\gaussp}[1]{\gamma_{#1}} 
\newcommand*{\Colon}[0]{\,:\,} 
\newcommand*{\dx}[0]{\mathrm{d}} 

\DeclareMathOperator{\diag}{diag}
\DeclareMathOperator{\Pro}{\mathbb{P}}
\DeclareMathOperator{\Exp}{\mathbb{E}}
\DeclareMathOperator{\Var}{Var}
\DeclareMathOperator{\Cov}{Cov}
\DeclareMathOperator{\vol}{vol}
\DeclareMathOperator{\Uni}{Unif}
\DeclareMathOperator{\Norm}{\mathcal{N}}
\DeclareMathOperator{\BigO}{\mathcal{O}}
\DeclareMathOperator{\smallO}{\mathit{o}}
\DeclareMathOperator{\BigTh}{\Theta}

\newtheorem{thm}{Theorem}[section]
\newtheorem{cor}[thm]{Corollary}
\newtheorem{lemma}[thm]{Lemma}

\newtheorem{propos}[thm]{Proposition}

\newtheorem{thmalpha}{Theorem}

\theoremstyle{definition}

\newtheorem{rmk}[thm]{Remark}

\allowdisplaybreaks

\begin{document}


\title{\textbf{Spectral flatness and the volume of\\intersections of \(p\)-ellipsoids}}

\vspace{0.5\baselineskip}

\author{Michael Juhos and Joscha Prochno}



\date{}

\maketitle

\begin{abstract}
\small
Motivated by classical works of Schechtman and Schmuckenschl\"ager on intersections of \(\ell_p\)-balls and recent ones in information-based complexity relating random sections of ellipsoids and the quality of random information in approximation problems, we study the threshold behavior of the asymptotic volume of intersections of generalized \(p\)-ellipsoids. The non-critical behavior is determined under a spectral flatness (Wiener entropy) condition on the semi-axes. In order to understand the critical case at the threshold, we prove a central limit theorem for \(q\)-norms of points sampled uniformly at random from a \(p\)-ellipsoid, which is obtained under Noether's condition on the semi-axes.

\vspace*{0.5\baselineskip}
\end{abstract}


\section{Introduction and main results}

The asymptotic geometry of \(\ell_p^n\)-balls has been studied intensively in the last decades and applications of probabilistic methods have proved to be very powerful for those studies. 
One of the earlier works is a well-known paper by Schechtman and Zinn~\cite{SchechtmanZinn} in which, motivated by a question of V.D. Milman, the authors studied the proportion of volume left in a volume normalized \(\ell_p^n\)-ball after removing a multiple of a volume normalized \(\ell_q^n\)-ball. A fundamental element in their approach is a probabilistic representation of the cone probability measure on an \(\ell_p^n\)-sphere (more precisely, on the positive orthant), which can be easily extended to a representation of the uniform distribution on the whole ball.
Independently, this probabilistic representation had been found by Rachev and R\"uschendorf~\cite{RR1991}, but with a different objective: they proved a Maxwell-principle for \(\ell_p^n\)-spheres, complementing the classical Poincar\'e-Maxwell-Borel lemma for the Euclidean sphere. This representation has again been put to use in a work of Schechtman and Schmuckenschläger~\cite{SS1991}, where the authors investigated the limit \(\lim_{n \to \infty} \vol_n(\D_p^n \cap t \D_q^n)\) and its dependence on the parameter \(t\in(0,\infty)\). Here, we denote by \(\D_p^n\) the volume normalized version of the closed unit ball \(\B_p^n\) of \(\ell_p^n\); recall that for \(0 < p \leq \infty\) and \(x = (x_i)_{i=1}^n \in \R^n\),
\begin{align*}
\lVert x \rVert_p :=%
\begin{cases}
\Bigl( \sum\limits_{i = 1}^n \lvert x_i \rvert^p \Bigr)^{1/p} & \colon p < \infty,\\ \max\limits_{1\leq i\leq n} \lvert x_i \rvert & \colon p = \infty
\end{cases}
\end{align*}
defines a quasi-norm on \(\R^n\), which is a norm whenever \(1 \leq p \leq \infty\).
Using a law of large numbers, they showed that the limit exists and is equal to either zero or one, depending on whether \(t\) is smaller or greater, respectively, than a certain threshold determined solely by the parameters \(p\) and \(q\). The asymptotic behavior at the threshold remained unanswered for nearly a decade until Schmuckenschläger~\cite{Schmu1998} used the Berry--Esseen theorem to show that the limit is \(1/2\) in the case \(p = \infty\), independent of \(q\). Another couple of years later Schmuckenschläger~\cite{Schmu2001} provided a complete answer by proving a central limit theorem for the \(q\)-norm of a vector sampled uniformly at random from the \(\ell_p^n\)-sphere; again the limit is \(1/2\) irrespective of \(p\) and \(q\).

The previous results have been revisited by Kabluchko, Prochno, and Th\"ale in \cite{KPT2019_I} (see also \cite{KPT2019_II}). Not only did they prove (multivariate) central limit theorems for \(q\)-norms of vectors distributed uniformly in \(\B_p^n\), but also non-central limit theorems (i.e., convergence in distribution to 
exponential and Gumbel distributions) for extreme values of \(p\) and \(q\); furthermore they showed that \(\lim_{n \to \infty} \vol_n(\D_p^n \cap t_n \D_q^n)\) can be any number in the interval \((0, 1)\) (not just $1/2$) with an appropriate sequence \((t_n)_{n\in\N}\) converging to the threshold. They also present results concerning the intersection of more than two balls, intersections of ``neighboring'' balls, and investigate large deviation principles for \(q\)-norms of random vectors sampled from \(\ell_p^n\)-balls, where the authors have taken up the thread laid by Gantert, Kim, and Ramanan~\cite{GKR2017}. The topic of large deviation principles in geometric settings has been researched much since then, but the present paper does not touch upon it.

Another notable generalization of the results of Schechtman and Schmuckenschläger has recently been accomplished by Kabluchko and Prochno~\cite{KP2021}: they have considered the intersection of Orlicz balls
\begin{equation*}
\B_M^n(nR) := \biggl\{ (x_1, \dotsc, x_n) \in \R^n \Colon \sum_{i = 1}^n M(x_i) \leq n R \biggr\},
\end{equation*}
where \(R \in (0, \infty)\), and \(M \colon \R \to [0, \infty)\) is an Orlicz-function, that is, an even, convex function with \(M(x) = 0\) iff \(x = 0\). Like in the case of \(\ell_p^n\)-balls there is a threshold for the parameter \(R\), depending only on the two Orlicz functions involved (one for each ball), such that the limit is either zero or one; the authors conjecture that the limit equals \(1/2\) at the threshold. Whereas the proof of Schechtman and Schmuckenschläger only needs a (weak) law of large numbers and builds upon the Schechtman--Zinn probabilistic representation, Kabluchko and Prochno require completely different and finer tools. The lack in a corresponding probabilistic representation for Orlicz balls is overcome by means of the maximum entropy principle from statistical mechanics and Petrov's sharp version of Cramér's large deviation theorem \cite{P1965}, which says that, under suitable conditions, the probability in the weak law of large numbers converges exponentially fast towards zero with a specified rate. 

Motivated by the previous body of research and by recent works in information-based complexity, where the asymptotic geometry of ellipsoids is related to questions about the quality of random information in approximation problems \cite{HKNPU2020, HKNPV2021,HPS2021} and because volumes of intersections appear naturally in tractability questions for multivariate numerical integration of certain classes of smooth functions \cite{HPU2019}, the present paper undertakes another generalization of the results of Schechtman and Schmuckenschläger, where, instead of the isotropic \(\ell_p^n\)-balls, we shall investigate \(p\)-ellipsoids, which are defined as follows: for \(0 < p \leq \infty\), \(n \in \N\), and \(\sigma = (\sigma_1, \dotsc, \sigma_n) \in (0, \infty)^n\), 
\begin{equation*}
\Ell_{p, \sigma}^n := \biggl\{(x_1, \dotsc, x_n) \in \R^n \Colon \sum_{i = 1}^n \Bigl\lvert \frac{x_i}{\sigma_i} \Bigr\rvert^p \leq 1 \biggr\}
\end{equation*}
if \(p < \infty\), or
\begin{equation*}
\Ell_{\infty, \sigma}^n := \biggl\{(x_1, \dotsc, x_n) \in \R^n \Colon \max_{1\leq i\leq n } \Bigl\lvert \frac{x_i}{\sigma_i} \Bigr\rvert \leq 1 \biggr\}
\end{equation*}
if \(p = \infty\), and we refer to \(\Ell_{p, \sigma}^n\) as the \emph{\(p\)-ellipsoid with semi-axes \(\sigma_1, \dotsc, \sigma_n\)}. We shall denote by \(\nEll_{p, \sigma}^n\) the unit volume dilation of \(\Ell_{p, \sigma}^n\) (i.e., \(\vol_n(\nEll_{p, \sigma}^n) = 1\)) and study the asymptotic of \(\lim_{n \to \infty} \vol_n(\nEll_{p, \sigma_n}^n \cap t \nEll_{q, \tau_n}^n)\), where \(p\), \(q\), and \(t\) are fixed, and where \(\sigma_n, \tau_n \in (0, \infty)^n\) for each \(n \in \N\). 

\subsection{The main results}
\label{sec:main_results}

We shall now present our main results on the threshold behavior of the asymptotic volume of intersections of \(p\)-ellipsoids. Note that by a simple linear transformation argument it is enough to consider the intersection of volume normalized \(\ell_p^n\)-balls with volume normalized \(q\)-ellipsoids (we shall elaborate in more detail in Remark~\ref{rmk:fqsigma_neq_zero} below). As we will see, the result in the non-critical case is proved under a so-called spectral flatness condition, also known as Wiener entropy, on the semi-axes of the ellipsoid.

Before we state this and the other results, let us briefly introduce some notation. For any \(p \in (0, \infty]\), the \emph{\(p\)-generalized Gaussian distribution} (\(p\)-Gaussian distribution for short) \(\gaussp{p}\) on \(\R\) is defined by
\begin{equation*}
\dif\gaussp{p}(x) :=%
\begin{cases}
\frac{1}{2 p^{1/p} \, \Gamma(\frac{1}{p} + 1)} \, \ee^{-\lvert x \rvert^p/p} \, \dif x & \colon p < \infty,\\
\frac{1}{2} \, 1_{[-1, 1]}(x) \, \dif x & \colon p = \infty.
\end{cases}
\end{equation*}
We introduce the following shorthands for the moments of \(\gaussp{p}\): let \(X \sim \gaussp{p}\) and \(q, r \in [0, \infty]\), then
\begin{equation*}
M_p(q) := \Exp[\lvert X \rvert^q], \qquad V_p(q) := \Var[\lvert X \rvert^q], \qquad C_p(q, r) := \Cov[\lvert X \rvert^q, \lvert X \rvert^r],
\end{equation*}
and we adopt the conventions \(M_\infty(\infty) := 0\), \(V_\infty(\infty) := 0\), and \(C_\infty(\infty, \infty) := C_\infty(\infty, q) := 0\); also note \(M_p(p) = 1\) and \(V_p(p) = p\) for \(p < \infty\). In what follows \(\bsigma\) shall always denote an infinite triangular array with positive real entries, namely \(\bsigma = ((\sigma_{n, i})_{i=1}^n)_{n \in \N} \in \prod_{n \in \N} (0, \infty)^n\). For \(n \in \N\) and \(p \in (0, \infty]\) let \(r_{n, p} := \vol_n(\B_p^n)^{-1/n}\) denote the radius of the volume-normalized \(\ell_p^n\)-ball \(\D_p^n\).

The first result is a generalization of a celebrated theorem for \(\ell_p^n\)-balls obtained by Schechtman and Schmuckenschläger in \cite{SS1991}.

\begin{thmalpha}\label{thm:threshold_theorem}
Let \(p \in (0, \infty]\) and \(q \in (0, \infty)\), and let $\bsigma = ((\sigma_{n, i})_{i=1}^n)_{n \in \N} \in \prod_{n \in \N} (0, \infty)^n$ be given. Assume that the limit
\begin{equation*}
F_{q, \bsigma} := \lim_{n \to \infty}\biggl[ \biggl( \prod_{i = 1}^n \sigma_{n, i} \biggr)^{\frac{1}{n}} \biggl( \frac{1}{n} \sum_{i = 1}^n \sigma_{n, i}^{-q} \biggr)^{\frac{1}{q}} \biggr]
\end{equation*}
exists in \([0, \infty]\) and that 
\begin{equation}\label{eq:sigma}
\lim_{n \to \infty} \biggl[ \biggl( \prod_{i = 1}^n \sigma_{n, i} \biggr)^{\frac{2 q}{n}} \frac{1}{n^2} \sum_{i = 1}^n \sigma_{n, i}^{-2 q} \biggr] = 0.
\end{equation}
Define 
\begin{equation*}
A_{p, q} := M_p(q)^{-\frac{1}{q}} \lim_{n \to \infty} \frac{n^{-\frac{1}{q}} \, r_{n, q}}{n^{-\frac{1}{p}} \, r_{n, p}} =%
\begin{cases}
\frac{\Gamma(1 + \frac{1}{p})^{1 + 1/q}}{\Gamma(1 + \frac{1}{q}) \Gamma(\frac{q + 1}{p})^{1/q}} \, \ee^{\frac{1}{p} - \frac{1}{q}} \, \bigl( \frac{p}{q} \bigr)^{\frac{1}{q}} & \colon p < \infty\\
\frac{1}{\Gamma(1 + \frac{1}{q})} \bigl( \frac{q + 1}{q \ee} \bigr)^{\frac{1}{q}} & \colon p = \infty.
\end{cases}
\end{equation*}
Then, for all \(t \in [0, \infty)\),
\begin{equation*}
\lim_{n \to \infty} \vol_n\big(\D_p^n \cap t \nEll_{q, \sigma_n}^n\big) =%
\begin{cases}
0 & \colon t A_{p, q} < F_{q, \bsigma}\\
1 & \colon t A_{p, q} > F_{q, \bsigma},
\end{cases}
\end{equation*}
where we have set \(\sigma_n := (\sigma_{n, i})_{i=1}^n\).
\end{thmalpha}

\begin{rmk}\label{rmk:fqsigma_neq_zero}
(1)~Theorem~\ref{thm:threshold_theorem} and Corollary~\ref{cor:threshold_theorem_crit} (below) also cover the case of the intersection of a \(p\)-ellipsoid with a \(q\)-ellipsoid, because
\begin{align*}
\vol_n(\nEll_{p, \sigma}^n \cap t \nEll_{q, \tau}^n) &= \vol_n(\det(\Sigma)^{-1/n} \Sigma \D_p^n \cap t \det(T)^{-1/n} T \D_q^n)\\
&= \vol_n(\D_p^n \cap t \det(\inv{\Sigma} T)^{-1/n} (\inv{\Sigma} T) \D_q^n)\\
&= \vol_n(\D_p^n \cap t \nEll_{q, \rho}^n),
\end{align*}
where \(\sigma = (\sigma_i)_{i = 1}^n\), \(\tau = (\tau_i)_{i = 1}^n\), \(\Sigma = \diag(\sigma)\), \(T = \diag(\tau)\), and \(\rho = (\inv{\sigma_i} \tau_i)_{i = 1}^n\). The sufficient conditions must be fulfilled by \(\rho\) then.

\noindent
(2)~By the inequality of the arithmetic and geometric means we see
\begin{equation*}
\biggl( \prod_{i = 1}^n \sigma_{n, i} \biggr)^{\frac{1}{n}} = \biggl( \prod_{i = 1}^n \sigma_{n, i}^{-q} \biggr)^{\frac{1}{n} \bigl( -\frac{1}{q} \bigr)} \geq \biggl( \frac{1}{n} \sum_{i = 1}^n \sigma_{n, i}^{-q} \biggr)^{-\frac{1}{q}},
\end{equation*}
where equality holds true iff \(\sigma_{n, 1} = \dotsb = \sigma_{n, n}\); hence one sees \(F_{q, \bsigma} \geq 1\) actually.
\end{rmk}

Our second theorem is a central limit theorem (CLT) for a properly scaled \(q\)-norm of a vector sampled uniformly at random from a \(p\)-ellipsoid and obtained under the so-called Noether condition on the semi-axes. It thus partly generalizes \cite[Theorem~1.1]{KPT2019_I}, which treats several distinct values of \(q\) simultaneously and in addition contains two non-central limit theorems. 

\begin{thmalpha}\label{thm:clt_ellipsoid}
Let \(p\), \(q\), \(\bsigma\), \(F_{q, \bsigma}\) be given as in Theorem~\ref{thm:threshold_theorem}. Furthermore, let \(\bsigma\) satisfy~\eqref{eq:sigma} and Noether's condition
\begin{equation}\label{eq:feller_sigma}
\lim_{n \to \infty} \frac{\max\limits_{1 \leq i \leq n} \sigma_{n, i}^{-2 q}}{\sum_{i = 1}^n \sigma_{n, i}^{-2 q}} = 0,
\end{equation}
and in the case \(p < \infty\) let the following limit exist,
\begin{equation*}
G_{q, \bsigma} := \lim_{n \to \infty} \frac{\sum_{i = 1}^n \sigma_{n, i}^{-q}}{\sqrt{n} \bigl( \sum_{i = 1}^n \sigma_{n, i}^{-2 q} \bigr)^{1/2}}.
\end{equation*}
For each \(n\in\N\), let \(Z_n \sim \Uni(\B_p^n)\) and \(\Sigma_n := \diag((\sigma_{n, i})_{i=1}^n)\). If \(p \neq q\) or \(G_{q, \bsigma} < 1\), then,
\begin{equation*}
\Biggl( \frac{\sum_{i = 1}^n \sigma_{n, i}^{-q}}{\bigl( \sum_{i = 1}^n \sigma_{n, i}^{-2 q} \bigr)^{1/2}} \biggl( \frac{n^{1/p} \, \lVert \inv{\Sigma_n} Z_n \rVert_q}{M_p(q)^{1/q} \, (\sum_{i = 1}^n \sigma_{n, i}^{-q})^{1/q}} - 1 \biggr) \Biggr)_{n \in \N} \xrightarrow{\text{\upshape d}} \Norm(0, s^2)
\end{equation*}
with
\begin{equation*}
s^2 := \frac{V_p(q)}{q^2 M_p(q)^2} - \frac{2 C_p(p, q) G_{q, \bsigma}^2}{p q M_p(q)} + \frac{G_{q, \bsigma}^2}{p}.
\end{equation*}
\end{thmalpha}

\begin{rmk}\label{rmk:gqsigma}
(1)~If it exists, we know \(G_{q, \bsigma} \in [0, 1]\) because of the comparison of the \(\ell_1\)- and \(\ell_2\)-norms on \(\R^n\). Moreover, if \(F_{q, \bsigma} < \infty\), then because of
\begin{equation*}
\frac{\sum_{i = 1}^n \sigma_{n, i}^{-q}}{\sqrt{n} \bigl( \sum_{i = 1}^n \sigma_{n, i}^{-2 q} \bigr)^{1/2}} = \frac{\bigl( \prod_{i = 1}^n \sigma_{n, i} \bigr)^{q/n} \, \frac{1}{n} \sum_{i = 1}^n \sigma_{n, i}^{-q}}{\bigl( \bigl( \prod_{i = 1}^n \sigma_{n, i} \bigr)^{2 q/n} \, \frac{1}{n} \sum_{i = 1}^n \sigma_{n, i}^{-2 q} \bigr)^{1/2}},
\end{equation*}
\(G_{q, \bsigma}\) exists if and only if
\begin{equation*}
F_{2 q, \bsigma}^{2 q} = \lim_{n \to \infty} \biggl( \prod_{i = 1}^n \sigma_{n, i} \biggr)^{\frac{2 q}{n}} \, \frac{1}{n} \sum_{i = 1}^n \sigma_{n, i}^{-2 q}
\end{equation*}
exists, and then
\begin{equation*}
G_{q, \bsigma} = \frac{F_{q, \bsigma}^q}{F_{2 q, \bsigma}^q}.
\end{equation*}
Note that by Remark~\ref{rmk:fqsigma_neq_zero}, (2), \(F_{2 q, \bsigma} \geq 1\), and hence division by zero is precluded.

\noindent
(2)~Let \(X \sim \gaussp{p}\), then we can express
\begin{equation*}
s^2 = \Var\Bigl[ \frac{\lvert X \rvert^q}{q M_p(q)} - \frac{G_{q, \bsigma}^2 \lvert X \rvert^p}{p} \Bigr] + \frac{G_{q, \bsigma}^2 (1 - G_{q, \bsigma}^2)}{p}
\end{equation*}
and therefore \(s = 0\) if and only if \(p = q\) and \(G_{q, \bsigma} = 1\), which is what we have excluded.
\end{rmk}

The CLT in Theorem~\ref{thm:clt_ellipsoid} allows us to determine the asymptotic behavior in the critical case at the threshold \(t_{\text{\upshape crit}} = \inv{A_{p, q}} F_{q, \bsigma}\), which is not covered by Theorem~\ref{thm:threshold_theorem}. It presents the generalization of Theorem~2.1 in Schmuckenschläger~\cite{Schmu1998} for \(p = \infty\) and of Theorem~3.2 in Schmuckenschläger~\cite{Schmu2001} for \(p < \infty\).

\begin{cor}\label{cor:threshold_theorem_crit}
Let \(p\), \(q\), \(\bsigma\), \(F_{q, \bsigma}\), \(G_{q, \bsigma}\), \(s\) be given as in Theorem~\ref{thm:clt_ellipsoid}, but with \(F_{q, \bsigma} < \infty\). Define, for \(n \in \N\),
\begin{equation*}
h_n := \frac{\bigl( \prod_{i = 1}^n \sigma_{n, i} \bigr)^{1/n} \bigl( \frac{1}{n} \sum_{i = 1}^n \sigma_{n, i}^{-q} \bigr)^{1/q}}{F_{q, \bsigma}}.
\end{equation*}
Then in the critical case \(t_{\text{\upshape crit}} = \inv{A_{p, q}} \, F_{q, \bsigma}\), if
\begin{equation*}
z := \lim_{n \to \infty} \frac{\sum_{i = 1}^n \sigma_{n, i}^{-q}}{\bigl( \sum_{i = 1}^n \sigma_{n, i}^{-2 q} \bigr)^{1/2}} (h_n - 1) \in [-\infty, \infty]
\end{equation*}
exists, it holds true that
\begin{equation*}
\lim_{n \to \infty} \vol_n(\D_p^n \cap t_{\text{\upshape crit}} \nEll_{q, \sigma_n}^n) = \Phi\Bigl( -\frac{z}{s} \Bigr),
\end{equation*}
where \(\Phi\) is the CDF of the standard normal distribution.
\end{cor}

\begin{rmk}\label{rmk:two_ellipsoids}
Different from the case of balls treated in~\cite[Corollary~2.1]{KPT2019_I}, for ellipsoids it makes sense to seperate the critical from the non-critical case for \(t\) (the present Corollary~\ref{cor:threshold_theorem_crit} and Theorem~\ref{thm:threshold_theorem}, respectively), because the latter needs weaker premises compared to the former as is evident from our formulation above.
\end{rmk}

\subsection*{Organization of the paper}

The rest of this paper is structured as follows. In Section~\ref{sec:prelim} we discuss notation and preliminaries. Examples of semi-axes satisfying the assumptions in our theorems are presented in Section~\ref{sec:examples}. In Sections~\ref{sec:proofs} and~\ref{sec:proof_spec_sigma} we provide the proofs for our main results and of the examples, respectively.


\section{Notation and preliminaries}\label{sec:prelim}

We shall now present the notation used throughout this paper followed by a short subsection discussing the uniform distribution on \(p\)-ellipsoids.

\subsection{Notation}

The Landau-symbols will be used in the proofs: for real sequences \((a_n)_{n\in\N}\) and \((b_n)_{n\in\N}\) define
\begin{align*}
a_n = \BigO(b_n) &:\Longleftrightarrow \exists M \in (0, \infty) \, \exists n_0 \geq 1 \, \forall n \geq n_0 \colon \lvert a_n \rvert \leq M \lvert b_n \rvert,\\
a_n = \smallO(b_n) &:\Longleftrightarrow \forall \eps \in (0, \infty) \, \exists n_0 \geq 1 \, \forall n \geq n_0 \colon \lvert a_n \rvert \leq \eps \lvert b_n \rvert,\\
a_n = \BigTh(b_n) &:\Longleftrightarrow \exists m, M \in (0, \infty) \, \exists n_0 \geq 1 \, \forall n \geq n_0 \colon m \lvert b_n \rvert \leq \lvert a_n \rvert \leq M \lvert b_n \rvert.
\end{align*}
Mostly we will use the symbols \(\BigO\), \(\smallO\), and \(\BigTh\) loosely as stand-ins for sequences with the respective property, e.g.,\ \(1 + \BigO(\frac{1}{n})\) is to be understood as \(1 + a_n\) with some \(a_n = \BigO(\frac{1}{n})\). In particular \(\BigO(1)\) stands for a bounded sequence, \(\smallO(1)\) for a null-sequence, and \(\BigTh(1)\) for a bounded sequence which is also bounded away from zero.

\subsection{\(p\)-ellipsoids and the uniform distribution on them}

We assume that all random variables are defined on a common probability space \((\Omega, \mathcal{A}, \Pro)\); the expected value, variance, and covariance with respect to \(\Pro\) are denoted by \(\Exp\), \(\Var\), and \(\Cov\), respectively.

Recall the definition of the \(p\)-Gaussian distribution \(\gaussp{p}\) and its moments at the beginning of Section~\ref{sec:main_results}. The uniform distribution on a set \(A \subset \R^n\) with positive volume is denoted by \(\Uni(A)\), and the (multivariate) normal distribution with mean (vector) \(\mu\) and variance (covariance matrix) \(v\) by \(\Norm(\mu, v)\); note \(\Norm(0, 1) = \gaussp{2}\). Equality and convergence in distribution are indicated by \(\stackrel{\mathrm{d}}{=}\) and \(\xrightarrow{\mathrm{d}}\), respectively.

Let us continue with two important properties of \(p\)-ellipsoids. The first concerns the relation between balls and ellipsoids. Set \(\Sigma := \diag(\sigma) \in \R^{n \times n}\). Obviously by definition \(x \in \Ell_{p, \sigma}^n\) if and only if \(\inv{\Sigma} x \in \B_p^n\), and therefore \(\Ell_{p, \sigma}^n = \Sigma \B_p^n\). From this follows \(\vol_n(\Ell_{p, \sigma}^n) = \det(\Sigma) \vol_n(\B_p^n)\). Futhermore, we have
\begin{equation}\label{eq:nEll}
\nEll_{p, \sigma}^n = \frac{\Ell_{p, \sigma}^n}{\vol_n(\Ell_{p, \sigma}^n)^{1/n}} = \frac{\Sigma \B_p^n}{\det(\Sigma)^{1/n} \, \vol_n(\B_p^n)^{1/n}} = \frac{\Sigma \D_p^n}{\det(\Sigma)^{1/n}}.
\end{equation}
Equation~\eqref{eq:nEll} also implies that \(\nEll_{p, \sigma}^n\) is invariant under scaling of \(\sigma\), that is, for any \(c \in (0, \infty)\) the identity \(\nEll_{p, c \sigma}^n = \nEll_{p, \sigma}^n\) holds true, where \(c \sigma := (c \sigma_i)_{i=1}^n\).

The second statement relates the uniform distributions on balls and ellipsoids: let \(X\) be an \(\R^n\)-valued random variable, then \(X \sim \Uni(\B_p^n)\) iff \(\Sigma X \sim \Uni(\Ell_{p, \sigma}^n)\).
Indeed, let \(X \sim \Uni(\B_p^n)\) and let \(A \subset \R^n\) be a Borel-set, then
\begin{align*}
\Pro[\Sigma X \in A] &= \Pro[X \in \inv{\Sigma} A] = \frac{\vol_n(\inv{\Sigma} A \cap \B_p^n)}{\vol_n(\B_p^n)} = \frac{\vol_n(\inv{\Sigma} A \cap \inv{\Sigma} \Ell_{p, \sigma}^n)}{\vol_n(\inv{\Sigma} \Ell_{p, \sigma}^n)}\\
&= \frac{\vol_n(\inv{\Sigma} (A \cap \Ell_{p, \sigma}^n))}{\vol_n(\inv{\Sigma} \Ell_{p, \sigma}^n)} = \frac{\vol_n(A \cap \Ell_{p, \sigma}^n)}{\vol_n(\Ell_{p, \sigma}^n)},
\end{align*}
hence \(\Sigma X \sim \Uni(\Ell_{p, \sigma}^n)\) as claimed. The reverse direction follows similarly.

\section{Examples}\label{sec:examples}

We now provide two non-trivial examples of arrays of semi-axes, either of which generalizes the case of balls and yet can be shown directly to satisfy the conditions of Theorems~\ref{thm:threshold_theorem} and~\ref{thm:clt_ellipsoid}. Either example is simple in that the rows of the array arise as successive initial segments of a single sequence of positive numbers. We call \((\sigma_n)_{n \in \N} \subset (0, \infty)\) eventually periodic if and only if there exist \(n_0 \in \N_0\) and \(d \in \N\) such that \(\sigma_{n + d} = \sigma_{n}\) for all \(n \geq n_0 + 1\).

\begin{propos}\label{prop:spec_sigma}
Let \(p \in (0, \infty]\), \(q \in (0, \infty)\), let \((\sigma_n)_{n \in \N} \subset (0, \infty)\) and define \(\bsigma := ((\sigma_i)_{i=1}^n)_{n \in \N}\).
\begin{compactenum}[(a)]
\item Let \((\sigma_n)_{n \in \N}\) be eventually periodic. Then \(\bsigma\) satisfies the premises of Theorems~\ref{thm:threshold_theorem} and~\ref{thm:clt_ellipsoid} with
\begin{equation*}
F_{q, \bsigma} = \biggl( \prod_{i = 1}^{d} \sigma_{n_0 + i} \biggr)^{1/d} \biggl( \frac{1}{d} \sum_{i = 1}^{d} \sigma_{n_0 + i}^{-q} \biggr)^{1/q},
\end{equation*}
and Corollary~\ref{cor:threshold_theorem_crit} is valid with
\begin{equation*}
G_{q, \bsigma} = \frac{\sum_{i = 1}^{d} \sigma_{n_0 + i}^{-q}}{\sqrt{d} \bigl( \sum_{i = 1}^{d} \sigma_{n_0 + i}^{-2 q} \bigr)^{1/2}} \quad \text{and} \quad z = 0.
\end{equation*}
\item If \(\sigma_n = n^\alpha \log(n + 1)^\beta\) for some \(\alpha, \beta \in \R\) and for all \(n \in \N\), then \(\bsigma\) satisfies the premises of Theorem~\ref{thm:threshold_theorem} iff \(\alpha < \frac{1}{q}\) or \(\alpha = \frac{1}{q} \wedge \beta < 0\), and then
\begin{equation*}
F_{q, \bsigma} = \frac{1}{\ee^\alpha (1 - \alpha q)^{1/q}},
\end{equation*}
with the interpretation \(F_{q, \bsigma} = \infty\) for \(\alpha = \frac{1}{q}\).\\
Next, \(\bsigma\) satisfies the premises of Theorem~\ref{thm:clt_ellipsoid} and Corollary~\ref{cor:threshold_theorem_crit} iff \(\alpha < \frac{1}{2 q}\) or \(\alpha = \frac{1}{2 q} \wedge \beta \leq \frac{1}{2 q}\), and then there hold
\begin{equation*}
G_{q, \bsigma} = \frac{\sqrt{1 - 2 \alpha q}}{1 - \alpha q} \quad \text{and} \quad z = \begin{cases} -\infty & \text{if } \alpha \beta < 0, \\ 0 & \text{if } \beta = 0, \\ \infty & \text{if } \alpha \beta > 0 \text{ or } \alpha = 0 \wedge \beta \neq 0.\end{cases}
\end{equation*}
\end{compactenum}
\end{propos}

\noindent
For the proof of Proposition~\ref{prop:spec_sigma} we refer the reader to Section~\ref{sec:proof_spec_sigma}.

\begin{rmk}\label{rmk:semiaxes}
Ellipsoids with semi-axes like those in Proposition~\ref{prop:spec_sigma},~(b), have been studied in geometric functional analysis and information-based complexity, see, e.g., Hinrichs et~al.~\cite[Corollary~6]{HKNPV2021}. Moreover, they arise asymptotically as singular values of embeddings of Sobolev spaces; see, e.g., Novak and Wo\'zniakowski~\cite{NW2008}, Section~4.2.4 and Remark~4.43 and the references therein.
\end{rmk}


\section{Proofs of the main theorems}
\label{sec:proofs}

\subsection{Proof of Theorem~\ref{thm:threshold_theorem}}
\label{subsec:proof_threshold_theorem}

The proof of Theorem~\ref{thm:threshold_theorem} essentially follows along the lines set out in \cite{SS1991}, while making the necessary adaptations to account for the possibly unequal semi-axes.

The handling of \(\Uni(\B_p^n)\) for \(p < \infty\) is made amenable to probabilistic methods by the following representation, which has its roots in Schechtman and Zinn~\cite{SchechtmanZinn} and independently also in Rachev and Rüschendorf~\cite{RR1991} and which was further elaborated upon in Barthe et~al.~\cite{BartheGuedonEtAl}.

\begin{propos}\label{prop:schechtmanzinn}
A random vector \(Z\) has law \(\Uni(\B_p^n)\) if and only if there exist independent random variables \(U\) and \(X_1, \dotsc, X_n\) with
\begin{equation*}
U \sim \Uni([0, 1]) \quad \text{and} \quad X_i \sim \gaussp{p} \,\text{ for }\, 1\leq i \leq n
\end{equation*}
and such that
\begin{equation*}
Z \stackrel{\mathrm{d}}{=} U^{1/n} \, \frac{(X_1, \dotsc, X_n)}{\lVert (X_1, \dotsc, X_n) \rVert_p}.
\end{equation*}
\end{propos}

The main ingredient of the proof of Theorem~\ref{thm:threshold_theorem} is the following weak law of large numbers for triangular arrays.

\begin{propos}\label{prop:wlln_array}
Let \(((Y_{n, i})_{i=1}^n)_{n \in \N}\) be a triangular array of real-valued square-integrable random variables such that, for any \(n \in \N\), the random variables \(Y_{n, 1}, \dotsc, Y_{n, n}\) are uncorrelated, and let \((b_n)_{n \in \N} \subset (0, \infty)\) be a sequence. If the following condition holds true,
\begin{equation}
\lim_{n \to \infty} \frac{1}{b_n^2} \sum_{i = 1}^n \Var[Y_{n, i}] = 0,\label{eq:wlln}
\end{equation}
then
\begin{equation*}
\biggl( \frac{1}{b_n} \sum_{i = 1}^n (Y_{n, i} - \Exp[Y_{n, i}]) \biggr)_{n \in \N} \to 0 \text{ in probabilty.}
\end{equation*}
\end{propos}

\begin{proof}
A simple application of Chebyshev's
inequality yields, for any \(\eps \in (0, \infty)\),
\begin{align*}
\Pro\biggl[ \biggl\lvert \frac{1}{b_n} \sum_{i = 1}^n (Y_{n, i} - \Exp[Y_{n, i}]) \biggr\rvert \geq \eps \biggr] &\leq \frac{1}{b_n^2 \, \eps^2} \Var\biggl[ \sum_{i = 1}^n (Y_{n, i} - \Exp[Y_{n, i}]) \biggr]\\
&= \frac{1}{b_n^2 \, \eps^2} \sum_{i = 1}^n \Var[Y_{n, i}],
\end{align*}
and this immediately implies the conclusion.
\end{proof}


\begin{proof}[Proof of Theorem~\ref{thm:threshold_theorem}]
\textit{Case \(p < \infty\).} Let \(n \in \N\) and \(t \in [0, \infty)\). Let \(\vek{X}_n = (X_1, \dotsc, X_n) \sim \gaussp{p}^{\otimes n}\) and \(U \sim \Uni([0, 1])\) be independent, then \(r_{n, p} U^{1/n} \frac{\vek{X}_n}{\lVert \vek{X}_n \rVert_p} \sim \Uni(\D_p^n)\) and therefore, using \eqref{eq:nEll} and writing \(\Sigma_n := \diag(\sigma_n)\),
\begin{align}
\vol_n(\D_p^n \cap t \nEll_{q, \sigma_n}^n) &= \Pro\Bigl[ r_{n, p} U^{\frac{1}{n}} \frac{\vek{X}_n}{\lVert \vek{X}_n \rVert_p} \in t \det(\Sigma_n)^{-1/n} \Sigma_n \D_q^n \Bigr]\notag\\
&= \Pro\biggl[ \frac{r_{n, p} U^{\frac{1}{n}} \det(\Sigma_n)^{1/n} \inv{\Sigma_n} \vek{X}_n}{\lVert \vek{X}_n \rVert_p} \in t \D_q^n \biggr]\notag\\
&= \Pro\biggl[ \frac{r_{n, p} U^{\frac{1}{n}} \bigl( \sum_{i = 1}^n \bigl\lvert \frac{X_i}{\tau_{n, i}} \bigr\rvert^q \bigr)^{1/q}}{\bigl( \sum_{i = 1}^n \lvert X_i \rvert^p \bigr)^{1/p}} \leq t r_{n, q}\biggr]\notag\\
&= \Pro\biggl[ \frac{n^{-1/p} r_{n, p} U^{\frac{1}{n}} \bigl( \frac{1}{n} \sum_{i = 1}^n \bigl( \bigl\lvert \frac{X_i}{\tau_{n, i}} \bigr\rvert^q - \frac{M_p(q)}{\tau_{n, i}^q} \bigr) + \frac{\Exp[\lvert X_1 \rvert^q]}{n} \sum_{i = 1}^n \tau_{n, i}^{-q}\bigr)^{1/q}}{n^{-1/q} r_{n, q} \bigl( \frac{1}{n} \sum_{i = 1}^n \lvert X_i \rvert^p \bigr)^{1/p}} \leq t \biggr],\label{eq:pklinerinf}
\end{align}
where in the third line we have introduced \(\tau_{n, i} := \bigl( \prod_{j = 1}^n \sigma_{n, j} \bigr)^{-1/n} \, \sigma_{n, i}\) for better legibility. Now by the strong law of large numbers,
\begin{equation*}
\lim_{n \to \infty} \frac{1}{n} \sum_{i = 1}^n \lvert X_i \rvert^p = M_p(p) = 1 \text{ almost surely,}
\end{equation*}
and also \(\lim_{n \to \infty} U^{1/n} = 1\) almost surely. Furthermore, by Proposition~\ref{prop:wlln_array},
\begin{equation*}
\lim_{n \to \infty} \frac{1}{n} \sum_{i = 1}^n \Bigl( \Bigl\lvert \frac{X_i}{\tau_{n, i}} \Bigr\rvert^q - \frac{M_p(q)}{\tau_{n, i}^q} \Bigr) = 0 \text{ in probability;}
\end{equation*}
note that the condition \eqref{eq:sigma} on \(\bsigma\) corresponds precisely to the condition \eqref{eq:wlln} when applied to \(Y_{n, i} := \lvert \frac{X_i}{\tau_{n, i}} \rvert^q\) while plugging in the definition of \(\tau_{n, i}\). Therefore, the random variable on the left-hand side of \eqref{eq:pklinerinf} converges in probability, for \(n \to \infty\), towards \(\frac{F_{q, \bsigma}}{A_{p, q}}\). Since convergence in probability implies convergence in distribution, we get
\begin{equation*}
\lim_{n \to \infty} \vol_n(\D_p^n \cap t \nEll_{q, \sigma_n}^n) = \Pro\Bigl[ \frac{F_{q, \bsigma}}{A_{p, q}} \leq t \Bigr]
\end{equation*}
and the claim follows.

\textit{Case~\(p = \infty\).} Again let \(n \in \N\) and \(t \in [0, \infty)\). Let \(\vek{X}_n = (X_1, \dotsc, X_n) \sim \gaussp{\infty}^{\otimes n} = \Uni(\B_\infty^n)\), hence \(\frac{1}{2} \, \vek{X}_n \sim \Uni(\D_\infty^n)\). Again invoking \eqref{eq:nEll} and writing \(\Sigma_n\) and \(\tau_{n, i}\) as before, we obtain
\begin{align*}
\vol_n(\D_\infty^n \cap t \nEll_{q, \sigma_n}^n) &= \Pro\bigl[ \tfrac{1}{2} \, \vek{X}_n \in t \det(\Sigma_n)^{-\frac{1}{n}} \, \Sigma_n \D_q^n \bigr]\\
&= \Pro\biggl[ \frac{\bigl( \sum_{i = 1}^n \bigl\lvert \frac{X_i}{\tau_{n, i}} \bigr\rvert^q \bigr)^{1/q}}{2 r_{n, q}} \leq t \biggr]\\
&= \Pro\biggl[ \frac{\bigl( \frac{1}{n} \sum_{n = 1}^n \bigl( \bigl\lvert \frac{X_i}{\tau_{n, i}} \bigr\rvert^q - \frac{M_\infty(q)}{\tau_{n, i}^q} \bigr) + \frac{M_\infty(q)}{n} \sum_{i = 1}^n \tau_{n, i}^{-q} \bigr)^{1/q}}{2 n^{-1/q} \, r_{n, q}} \leq t \biggr].
\end{align*}
Employing essentially the same arguments as in the case \(p < \infty\), the left-hand side is seen to be converging in probability, for \(n \to \infty\), towards \(\frac{F_{q, \bsigma}}{A_{\infty, q}}\), and the conclusion follows as before.
\end{proof}

\subsection{Proofs of Theorem~\ref{thm:clt_ellipsoid} and Corollary~\ref{cor:threshold_theorem_crit}}
\label{subsec:proof_clt_ellipsoid}


The proof of Theorem~\ref{thm:clt_ellipsoid} is more involving and requires preparations which are dealt with in the subsequent lemmas. The general strategy underlying the proof is the same as developed by Kalbuchko, the second named author, and Th\"ale in~\cite[Theorem~1.1 (a)]{KPT2019_I}.

We are going to use the following multivariate version of the Lindeberg--Feller-CLT in the proof of Theorem~\ref{thm:clt_ellipsoid}. As usual, expectations of vectors and matrices are to be understood component-wise; \(\vek{0} \in \R^d\) is the zero-vector and \(I_d \in \R^{d \times d}\) the identity-matrix. For vectors \(x = (x_i)_{i = 1}^n, y = (y_i)_{i = 1}^n \in \R^n\) we define their tensor product by \(x \otimes y := (x_i y_j)_{i, j = 1}^n \in \R^{n \times n}\).

\begin{propos}\label{prop:clt_multivar}
Let \(((X_{n, i})_{i=1}^n)_{n \in \N}\) be an independent array of centered, square-integrable \(\R^d\)-valued random variables. Define \(S_n := \sum_{i = 1}^n \Exp[X_{n, i} \otimes X_{n, i}] \in \R^{d \times d}\) and \(s_n := \lambda_{\min}(S_n)^{1/2}\) (the square root of the smallest eigenvalue of \(S_n\)); assume \(s_n > 0\) for all \(n\) sufficiently large. If the following multivariate Lindeberg-condition is satisfied:
\begin{equation*}
\lim_{n \to \infty} \frac{1}{s_n^2} \sum_{i = 1}^n \Exp\bigl[ \lVert X_{n, i} \rVert_2^2 \, 1_{[\lVert X_{n, i} \rVert_2 \geq \eps s_n]} \bigr] = 0 \quad \text{for all} \quad \eps \in (0, \infty),
\end{equation*}
then
\begin{equation*}
\biggl( S_n^{-1/2} \sum_{i = 1}^n X_{n, i} \biggr)_{n \in \N} \xrightarrow{\mathrm{d}} \Norm(\vek{0}, I_d).
\end{equation*}
\end{propos}

The proof of Proposition~\ref{prop:clt_multivar} is accomplished via the following theorem, slightly adapted from Klenke~\cite{Klenke2014}, Theorem~15.56. Here \(\langle (x_i)_{i = 1}^n, (y_i)_{i = 1}^n \rangle := \sum_{i = 1}^n x_i y_i\) denotes the standard inner product on \(\R^n\).

\begin{thm}[Cramér--Wold]
A sequence \((X_n)_{n\in\N}\) of \(\R^d\)-valued random variables converges in distribution to the \(\R^d\)-valued random variable \(X\) if and only if, for any \(\lambda \in \R^d\) with \(\lVert \lambda \rVert_2 = 1\), \((\langle \lambda, X_n \rangle)_{n\in\N}\) converges in distribution to \(\langle \lambda, X \rangle\).
\end{thm}

\begin{rmk}
The restriction to \(\lVert \lambda \rVert_2 = 1\) in comparison to~\cite{Klenke2014} is immaterial: let \(\lambda \in \R^d\) be arbitrary, then in the case \(\lambda = \vek{0}\) we have \(\langle \vek{0}, X_n \rangle = \langle \vek{0}, X \rangle = 0\) and the convergence is immediate; in the case \(\lambda \neq \vek{0}\) note \(\bigl\lVert \frac{\lambda}{\lVert \lambda \rVert_2} \bigr\rVert_2 = 1\), hence \(\frac{\langle \lambda, X_n \rangle}{\lVert \lambda \rVert_2} = \bigl\langle \frac{\lambda}{\lVert \lambda \rVert_2}, X_n \bigr\rangle \xrightarrow[n \to \infty]{\mathrm{d}} \bigl\langle \frac{\lambda}{\lVert \lambda \rVert_2}, X \bigr\rangle = \frac{\langle \lambda, X \rangle}{\lVert \lambda \rVert_2}\), and because the map \(x \mapsto \lVert \lambda \rVert_2 x\) is continuous we infer \((\langle \lambda, X_n \rangle)_{n\in\N} \xrightarrow{\mathrm{d}} \langle \lambda, X \rangle\).
\end{rmk}

\begin{proof}[Proof of Proposition~\ref{prop:clt_multivar}]
Note that if \(Z \sim \Norm(\vek{0}, I_d)\), then \(\langle \lambda, Z \rangle \sim \Norm(0, 1)\) for every \(\lambda \in \R^d\) with \(\lVert \lambda \rVert_2 = 1\). By the Cramér--Wold-theorem it is therefore enough to show
\begin{equation*}
\biggl( \biggl\langle \lambda, S_n^{-1/2} \sum_{i = 1}^n X_{n, i} \biggr\rangle \biggr)_{n \in \N} \xrightarrow{\mathrm{d}} \Norm(0, 1) \quad \text{for all} \quad \lambda \in \R^d \text{ with } \lVert \lambda \rVert_2 = 1.
\end{equation*}
Let \(\lambda \in \R^d\) with \(\lVert \lambda \rVert_2 = 1\). We set \(Y_{n, i} := \langle \lambda, S_n^{-1/2} X_{n, i} \rangle\) and will show that the array \(((Y_{n, i})_{i=1}^n)_{n \in \N}\) satisfies Lindeberg's condition (in one dimension). Obviously the latter array is independent and centered with square-integrable entries, and for any \(n \in \N\) we have
\begin{equation*}
\Var\biggl[ \sum_{i = 1}^n Y_{n, i} \biggr] = \sum_{i = 1}^n \Var[Y_{n, i}] = \biggl\langle \lambda, S_n^{-1/2} \sum_{i = 1}^n \Exp[X_{n, i} \otimes X_{n, i}] S_n^{-1/2} \lambda \biggr\rangle = 1.
\end{equation*}
Hence, it suffices to prove, for any \(\eps > 0\),
\begin{equation*}
\lim_{n \to \infty} \sum_{i = 1}^n \Exp\bigl[ Y_{n, i}^2 \, 1_{[\lvert Y_{n, i} \rvert \geq \eps]} \bigr] = 0;
\end{equation*}
then \(\sum_{i = 1}^n Y_{n, i} = \bigl\langle \lambda, S_n^{-1/2} \sum_{i = 1}^n X_{n, i} \bigr\rangle \limd{n} \Norm(0, 1)\) it true, as desired. We can estimate
\begin{align*}
\lvert Y_{n, i} \rvert &= \lvert \langle \lambda, S_n^{-1/2} X_{n, i} \rangle \rvert \leq \lVert \lambda \rVert_2 \, \lVert S_n^{-1/2} \rVert_2 \, \lVert X_{n, i} \rVert_2 =\\
&= \lambda_{\max}(S_n^{-1/2}) \lVert X_{n, i} \rVert_2 = \lambda_{\min}(S_n)^{-1/2} \, \lVert X_{n, i} \rVert_2 = \frac{\lVert X_{n, i} \rVert_2}{s_n},
\end{align*}
where \(\lVert \cdot \rVert_2\) denotes the spectral norm of a matrix, which equals the greatest eigenvalue in the case of a positive semidefinite matrix, here denoted by \(\lambda_{\max}\). This leads to
\begin{equation*}
\sum_{i = 1}^n \Exp\bigl[ Y_{n, i}^2 \, 1_{[\lvert Y_{n, i} \rvert \geq \eps]} \bigr] \leq \frac{1}{s_n^2} \sum_{i = 1}^n \Exp\bigl[ \lVert X_{n, i} \rVert_2^2 \, 1_{[\lVert X_{n, i} \rVert_2 \geq \eps s_n]} \bigr]
\end{equation*}
for any \(\eps \in (0, \infty)\), and by the premises of the proposition this implies the claim.
\end{proof}


We introduce the following notation in preparation for the proof of Theorem~\ref{thm:clt_ellipsoid}: let \((X_n)_{n\in\N}\) be a sequence of independent, \(\gaussp{p}\)-distributed random variables. With these we set
\begin{equation*}
\xi_n := \frac{\sum_{i = 1}^n \sigma_{n, i}^{-q} \, (\lvert X_i \rvert^q - M_p(q))}{\bigl( V_p(q) \sum_{i = 1}^n \sigma_{n, i}^{-2 q} \bigr)^{1/2}} \quad \text{and} \quad \eta_n := \frac{\sum_{i = 1}^n (\lvert X_i \rvert^p - 1)}{\sqrt{n p}};
\end{equation*}
the \(\eta_n\)'s are going to be needed (and in fact are defined) only for \(p < \infty\).

\begin{lemma}\label{lem:xi_eta_clt}
Let the premises of Theorem~\ref{thm:clt_ellipsoid} hold. Then in the case \(p < \infty\) the following convergence is valid,
\begin{equation*}
\left( \begin{pmatrix} \xi_n \\ \eta_n \end{pmatrix} \right)_{n \in \N} \xrightarrow{\mathrm{d}} \begin{pmatrix} \xi \\ \eta \end{pmatrix} \sim \Norm\left( \begin{pmatrix} 0 \\ 0 \end{pmatrix}, \begin{pmatrix} 1 & \rho \\ \rho & 1 \end{pmatrix} \right),
\end{equation*}
where
\begin{equation*}
\rho = \frac{C_p(p, q) G_{q, \bsigma}}{\sqrt{p \, V_p(q)}}.
\end{equation*}
And in the case \(p = \infty\),
\begin{equation*}
(\xi_n)_{n \in \N} \xrightarrow{\mathrm{d}} \xi \sim \Norm(0, 1).
\end{equation*}
\end{lemma}

\begin{proof}
\textit{Case~\(p < \infty\):} We rewrite,
\begin{equation*}\label{eq:xieta1}
\begin{pmatrix} \xi_n \\ \eta_n \end{pmatrix} = \begin{pmatrix} \frac{1}{\sqrt{V_p(q)}} & 0 \\ 0 & \frac{1}{\sqrt{p}} \end{pmatrix} \sum_{i = 1}^n \theta_{n, i} \zeta_i,
\end{equation*}
where we have defined
\begin{equation*}
\theta_{n, i} := \begin{pmatrix} \frac{\sigma_{n, i}^{-q}}{(\sum_{j = 1}^n \sigma_{n, j}^{-2 q})^{1/2}} & 0 \\ 0 & \frac{1}{\sqrt{n}} \end{pmatrix} \quad \text{and} \quad \zeta_i := \begin{pmatrix} \lvert X_i \rvert^q - M_p(q) \\ \lvert X_i \rvert^p - 1 \end{pmatrix};
\end{equation*}
the diagonal matrices \(\theta_{n, i}\) always are regular, and the sequence \((\zeta_n)_{n \in \N}\) consists of independent and identically distributed random vectors. We are going to apply Proposition~\ref{prop:clt_multivar} to get a CLT for the array \(((\theta_{n, i} \zeta_i)_{i = 1}^n)_{n \in \N}\). The covariance matrix is
\begin{equation}\label{eq:xieta4}
S_n := \sum_{i = 1}^n \Exp\bigl[ (\theta_{n, i} \zeta_i) \otimes (\theta_{n, i} \zeta_i) \bigr] = \begin{pmatrix} V_p(q) & \frac{\sum_{i = 1}^n \sigma_{n, i}^{-q}}{(n \sum_{i = 1}^n \sigma_{n, i}^{-2 q})^{1/2}} \, C_p(p, q) \\ \frac{\sum_{i = 1}^n \sigma_{n, i}^{-q}}{(n \sum_{i = 1}^n \sigma_{n, i}^{-2 q})^{1/2}} \, C_p(p, q) & p \end{pmatrix},
\end{equation}
and its smallest eigenvalue is
\begin{equation*}
s_n^2 := \frac{p + V_p(q)}{2} - \Bigl[ \Bigl( \frac{p + V_p(q)}{2} \Bigr)^2 - p V_p(q) + g_n^2 \, C_p(p, q)^2 \Bigr]^{1/2},
\end{equation*}
where we have abbreviated \(g_n := \frac{\sum_{i = 1}^n \sigma_{n, i}^{-q}}{(n \sum_{i = 1}^n \sigma_{n, i}^{-2 q})^{1/2}}\). By definition we know \(G_{q, \bsigma} = \lim_{n \to \infty} g_n\), therefore also
\begin{align*}
s_0^2 &:= \lim_{n \to \infty} s_n^2 = \frac{p + V_p(q)}{2} - \Bigl[ \Bigl( \frac{p + V_p(q)}{2} \Bigr)^2 - p V_p(q) + G_{q, \bsigma}^2 \, C_p(p, q)^2 \Bigr]^{1/2}\\
&= \frac{p + V_p(q)}{2} - \Bigl[ \Bigl( \frac{p + V_p(q)}{2} \Bigr)^2 - (1 - \rho^2) p V_p(q) \Bigr]^{1/2}
\end{align*}
exists, and by the premises of Theorem~\ref{thm:clt_ellipsoid} we get \(\lvert \rho \rvert < 1\), hence \(s_0 > 0\). Now there exists \(n_1 \in \N\) such that \(s_n \geq \frac{1}{2} \, s_0\) for all \(n \geq n_1\). Next we have the estimate
\begin{equation}\label{eq:xieta2}
\lVert \theta_{n, i} \zeta_i \rVert_2^2 \leq \max\biggl\{ \frac{\sigma_{n, i}^{-2 q}}{\sum_{j = 1}^n \sigma_{n, j}^{-2 q}}, \frac{1}{n} \biggr\} \lVert \zeta_i \rVert_2^2 \leq \frac{\max\limits_{1 \leq j \leq n} \sigma_{n, j}^{-2 q}}{\sum_{j = 1}^n \sigma_{n, j}^{-2 q}} \, \lVert \zeta_i \rVert_2^2,
\end{equation}
where the second inequality is valid since, because of \(\sum_{i = 1}^n \frac{\sigma_{n, i}^{-2 q}}{\sum_{j = 1}^n \sigma_{n, j}^{-2 q}} = 1\), there is at least one index \(i \in \{1,\dots,n\}\) such that \(\frac{\sigma_{n, i}^{-2 q}}{\sum_{j = 1}^n \sigma_{n, j}^{-2 q}} \geq \frac{1}{n}\). We also get the similar estimate
\begin{equation}\label{eq:xieta3}
\lVert \theta_{n, i} \zeta_i \rVert_2^2 \leq \biggl( \frac{\sigma_{n, i}^{-2 q}}{\sum_{i = 1}^n \sigma_{n, j}^{-2 q}} + \frac{1}{n} \biggr) \lVert \zeta_i \rVert_2^2.
\end{equation}
By its definition, \(\lVert \zeta_1 \rVert_2^2\) is integrable and hence almost surely finite, and the family of events \(([\lVert \zeta_1 \rVert_2^2 \geq x^2])_{x \in (0, \infty)}\) is monotonically decreasing, which implies
\begin{equation*}
\lim_{x \to \infty} 1_{[\lVert \zeta_1 \rVert_2^2 \geq x^2]} = 1_{[\lVert \zeta_1 \rVert_2^2 = \infty]} = 0 \text{ almost surely.}
\end{equation*}
Dominated convergence then leads to
\begin{equation*}
\lim_{x \to \infty} \Exp\bigl[ \lVert \zeta_1 \rVert_2^2 \, 1_{[\lVert \zeta_1 \rVert_2^2 \geq x^2]} \bigr] = 0.
\end{equation*}
Therefore, for any \(\eps \in (0, \infty)\) there exists \(x_0 \in (0, \infty)\) such that, for any \(x \geq x_0\),
\begin{equation*}
\Exp\bigl[ \lvert \zeta_1 \rVert_2^2 \, 1_{[\lvert \zeta_1 \rVert_2^2 \geq x^2]} \bigr] < \frac{\eps s_0^2}{8}.
\end{equation*}
Now concerning the multivariate Lindeberg's condition let \(\delta, \eps \in (0, \infty)\) and let \(x \geq x_0\) as before. By Noether's condition~\eqref{eq:feller_sigma} there exists \(n_2 \in \N\) such that, for all \(n \geq n_2\),
\begin{equation*}
\frac{\max\limits_{1 \leq i \leq n} \sigma_{n, i}^{-2 q}}{\sum_{i = 1}^n \sigma_{n, i}^{-2 q}} \leq \frac{\delta^2 \, s_0^2}{4 x^2}.
\end{equation*}
Let \(n \geq \max\{n_1, n_2\}\). Together with estimate~\eqref{eq:xieta2} the last display gives
\begin{equation*}
[\lVert \theta_{n, i} \zeta_i \rVert_2^2 \geq \delta^2 \, s_n^2] \subset [\lVert \zeta_i \rVert_2^2 \geq x^2]
\end{equation*}
for any \(i \in \{1,\dots,n\}\), which finally yields, employing also~\eqref{eq:xieta3},
\begin{align*}
\frac{1}{s_n^2} \sum_{i = 1}^n \Exp\bigl[ \lVert \theta_{n, i} \zeta_i \rVert_2^2 \, 1_{[\lVert \theta_{n, i} \zeta_i \rVert_2^2 \geq \delta^2 \, s_n^2]} \bigr] &\leq \frac{4}{s_0^2} \sum_{i = 1}^n \biggl( \frac{\sigma_{n, i}^{-2 q}}{\sum_{j = 1}^n \sigma_{n, j}^{-2 q}} + \frac{1}{n} \biggr) \Exp\bigl[ \lVert \zeta_1 \rVert_2^2 \, 1_{[\lVert \zeta_1 \rVert_2^2 \geq x^2]} \bigr]\\
&< \frac{4}{s_0^2} \cdot (1 + 1) \cdot \frac{\eps s_0^2}{8} = \eps,
\end{align*}
and this proves that the array \(((\theta_{n, i} \zeta_i)_{i=1}^n)_{n \in \N}\) satisfies Lindeberg's condition. Now from Proposition~\ref{prop:clt_multivar} we conclude
\begin{equation*}
\biggl( S_n^{-1/2} \sum_{i = 1}^n \theta_{n, i} \zeta_i \biggr)_{n\in\N} \xrightarrow{\mathrm{d}} Z \sim \Norm(\vek{0}, I_2).
\end{equation*}
From the definition of \(S_n\) in~\eqref{eq:xieta4} we see that \(S_0 := \lim_{n \to \infty} S_n\) exists, and therefore via Slutsky's theorem we obtain
\begin{equation*}
\begin{pmatrix} \xi_n \\ \eta_n \end{pmatrix} = \begin{pmatrix} \frac{1}{\sqrt{V_p(q)}} & 0 \\ 0 & \frac{1}{\sqrt{p}} \end{pmatrix} S_n^{1/2} \, S_n^{-1/2} \sum_{i = 1}^n \theta_{n, i} \zeta_i \xrightarrow[n \to \infty]{\mathrm{d}} \begin{pmatrix} \frac{1}{\sqrt{V_p(q)}} & 0 \\ 0 & \frac{1}{\sqrt{p}} \end{pmatrix} S_0^{1/2} \, Z.
\end{equation*}
(Recall Slutsky's theorem: given sequences \((X_n)_{n \in \N}\) and \((Y_n)_{n \in \N}\) of random variables and a continuous function \(f\), then \((f(X_n, Y_n))_{n \in \N}\) converges to \(f(X, y)\) in distribution whenever \((X_n)_{n \in \N}\) converges to the random variable \(X\) in distribution and \((Y_n)_{n \in \N}\) converges to the constant \(y\) in probability; in particular the latter is fulfilled if \((Y_n)_{n \in \N}\) is a convergent deterministic sequence.)

The limit random vector follows a normal distribution with expectation \(\vek{0}\) and covariance-matrix
\begin{align*}
&\quad \begin{pmatrix} \frac{1}{\sqrt{V_p(q)}} & 0 \\ 0 & \frac{1}{\sqrt{p}} \end{pmatrix} S_0^{1/2} \, I_2 S_0^{1/2} \begin{pmatrix} \frac{1}{\sqrt{V_p(q)}} & 0 \\ 0 & \frac{1}{\sqrt{p}} \end{pmatrix}\\
&= \begin{pmatrix} \frac{1}{\sqrt{V_p(q)}} & 0 \\ 0 & \frac{1}{\sqrt{p}} \end{pmatrix} \begin{pmatrix} V_p(q) & G_{q, \bsigma} C_p(p, q) \\ G_{q, \bsigma} C_p(p, q) & p \end{pmatrix} \begin{pmatrix} \frac{1}{\sqrt{V_p(q)}} & 0 \\ 0 & \frac{1}{\sqrt{p}} \end{pmatrix} \\
&= \begin{pmatrix} 1 & \frac{G_{q, \bsigma} C_p(p, q)}{\sqrt{p V_p(q)}} \\ \frac{G_{q, \bsigma} C_p(p, q)}{\sqrt{p V_p(q)}} & 1 \end{pmatrix}\\
&= \begin{pmatrix} 1 & \rho \\ \rho & 1 \end{pmatrix},
\end{align*}
as claimed.

\textit{Case~\(p = \infty\):} We do not spell out any details here, because the proof is analogous to the previous setup, apart from the fact that only the first component of every vector and the \((1, 1)\)-entry of every matrix is taken into consideration. In that case we have \(S_n = S_0 = s_n^2 = s_0^2 = V_p(q)\) throughout, and inequality~\eqref{eq:xieta3} reduces to the equality \(\lvert \theta_{n, i} \zeta_i \rvert^2 = \frac{\sigma_{n, i}^{-2 q}}{\sum_{j = 1}^n \sigma_{n, j}^{-2 q}} \, \lvert \zeta_i \rvert^2\).
\end{proof}

We are also going to use the following result, which allows the passing from convergence in distribution to almost sure convergence; it is commonly known as Skorokhod-representation or Skorokhod(--Dudley)-device and can be found, e.g., in~\cite{Klenke2014}, Theorem~17.56.

\begin{lemma}
Let \(E\) be a Polish space and let \((X_n)_{n\in\N}\) be a sequence of \(E\)-valued random variables converging in distribution to the \(E\)-valued random variable \(X_0\). Then there exist a probability space \((\widetilde{\Omega}, \widetilde{\mathcal{A}}, \widetilde{\Pro})\) and \(E\)-valued random variables \(\widetilde{X}_0, \widetilde{X}_1, \widetilde{X}_2, \dotsc\) defined on \(\widetilde{\Omega}\) such that
\begin{equation*}
\widetilde{X}_n \stackrel{\mathrm{d}}{=} X_n \quad \text{for all } n \in \N_0
\end{equation*}
and, as \(n\to\infty\),
\begin{equation*}
\widetilde{X}_n \to \widetilde{X}_0 \quad \text{\(\widetilde{\Pro}\)-almost surely.}
\end{equation*}
\end{lemma}

\begin{proof}[Proof of Theorem~\ref{thm:clt_ellipsoid}]
\textit{Case~\(p < \infty\):} With Proposition~\ref{prop:schechtmanzinn} we get (write \(\vek{X}_n := (X_i)_{i=1}^n\))
\begin{align*}
\lVert \inv{\Sigma_n} Z_n \rVert_q &\stackrel{\text{\upshape d}}{=} U^{\frac{1}{n}} \, \frac{\lVert \inv{\Sigma_n} \vek{X}_n \rVert_q}{\lVert \vek{X}_n \rVert_p} = U^{\frac{1}{n}} \, \frac{\bigl( \sum_{i = 1}^n \bigl\lvert \frac{X_i}{\sigma_{n, i}} \bigr\rvert^q \bigr)^{1/q}}{\bigl( \sum_{i = 1}^n \lvert X_i \rvert^p \bigr)^{1/p}} =\\
&= U^{\frac{1}{n}} \, \frac{\bigl( V_p(q)^{1/2} \, \bigl( \sum_{i = 1}^n \sigma_{n, i}^{-2 q} \bigr)^{1/2} \, \xi_n + M_p(q) \sum_{i = 1}^n \sigma_{n, i}^{-q} \bigr)^{1/q}}{(\sqrt{n p} \, \eta_n + n)^{1/p}}\\
&= U^{\frac{1}{n}} \, \frac{M_p(q)^{1/q} \, \bigl( \sum_{i = 1}^n \sigma_{n, i}^{-q} \bigr)^{1/q}}{n^{1/p}} \, \frac{\Bigl( 1 + \frac{V_p(q)^{1/2}}{M_p(q)} \, \frac{(\sum_{i = 1}^n \sigma_{n, i}^{-2 q})^{1/2}}{\sum_{i = 1}^n \sigma_{n, i}^{-q}} \, \xi_n \Bigr)^{1/q}}{\bigl( 1 + \sqrt{\frac{p}{n}} \, \eta_n \bigr)^{1/p}}\\
&= U^{\frac{1}{n}} \, \frac{M_p(q)^{1/q} \, \bigl( \sum_{i = 1}^n \sigma_{n, i}^{-q} \bigr)^{1/q}}{n^{1/p}} \, F\biggl( \frac{\bigl( \sum_{i = 1}^n \sigma_{n, i}^{-2 q} \bigr)^{1/2}}{\sum_{i = 1}^n \sigma_{n, i}^{-q}} \, \xi_n, \frac{\eta_n}{\sqrt{n}} \biggr),
\end{align*}
where we have defined the function
\begin{equation*}
F(x, y) := \frac{\bigl( 1 + \frac{V_p(q)^{1/2}}{M_p(q)} \, x \bigr)^{1/q}}{(1 + \sqrt{p} \, y)^{1/p}}.
\end{equation*}
By the Skorokhod-representation there exist a probability-space \((\widetilde{\Omega}, \widetilde{\mathcal{A}}, \widetilde{\Pro})\) and random variables \(\widetilde{\xi}_n\), \(\widetilde{\eta}_n\) for \(n \in \N\) and \(\widetilde\xi\), \(\widetilde{\eta}\) defined on \(\widetilde{\Omega}\) such that
\begin{equation*}
(\widetilde{\xi}_n, \widetilde{\eta}_n) \stackrel{\text{\upshape d}}{=} (\xi_n, \eta_n) \quad \text{and} \quad (\widetilde{\xi}, \widetilde{\eta}) \stackrel{\text{\upshape d}}{=} (\xi, \eta)
\end{equation*}
and such that
\begin{equation*}
\bigl( (\widetilde{\xi}_n, \widetilde{\eta}_n) \bigr)_{n \in \N} \to (\widetilde{\xi}, \widetilde{\eta}) \quad \text{\(\widetilde{\Pro}\)-a.s.}
\end{equation*}
From condition~\eqref{eq:sigma} we obtain, writing \(\tau_{n, i}\) like in the proof of Theorem~\ref{thm:threshold_theorem} and also respecting~Remark~\ref{rmk:fqsigma_neq_zero},
\begin{equation*}
\frac{\bigl( \sum_{i = 1}^n \sigma_{n, i}^{-2 q} \bigr)^{1/2}}{\sum_{i = 1}^n \sigma_{n, i}^{-q}} = \frac{\bigl( \frac{1}{n^2} \sum_{i = 1}^n \tau_{n, i}^{-2 q} \bigr)^{1/2}}{\frac{1}{n} \sum_{i = 1}^n \tau_{n, i}^{-q}} \xrightarrow[n \to \infty]{} \frac{0}{F_{q, \bsigma}} = 0,
\end{equation*}
and therefore \(\Bigl( \frac{(\sum_{i = 1}^n \sigma_{n, i}^{-2 q})^{1/2}}{\sum_{i = 1}^n \sigma_{n, i}^{-q}} \, \widetilde{\xi}_n \Bigr)_{n \in \N} \to 0\) and \(\bigl( \frac{\widetilde{\eta}_n}{\sqrt{n}} \bigr)_{n \in \N} \to 0\), either convergence being almost sure. A Taylor expansion of \(F\) around \((0, 0)\) yields
\begin{equation*}
F(x, y) = 1 + \frac{V_p(q)^{1/2} \, x}{q M_p(q)} - \frac{y}{\sqrt{p}} + \BigO(x^2 + y^2).
\end{equation*}
Furthermore, we know \(U^{1/n} = \ee^{\log(U)/n} = 1 + \BigO(\frac{1}{n})\) almost surely. Putting things together, we get
\begin{align*}
\lVert \inv{\Sigma_n} Z_n \rVert_q &\stackrel{\text{\upshape d}}{=} U^{\frac{1}{n}} \, \frac{M_p(q)^{1/q} \, \bigl( \sum_{i = 1}^n \sigma_{n, i}^{-q} \bigr)^{1/q}}{n^{1/p}} \, F\biggl( \frac{\bigl( \sum_{i = 1}^n \sigma_{n, i}^{-2 q} \bigr)^{1/2}}{\sum_{i = 1}^n \sigma_{n, i}^{-q}} \, \widetilde{\xi}_n, \frac{\widetilde{\eta}_n}{\sqrt{n}} \biggr) =\\
&= \frac{M_p(q)^{1/q} \, \bigl( \sum_{i = 1}^n \sigma_{n, i}^{-q} \bigr)^{1/q}}{n^{1/p}} \, \biggl( 1 + \frac{V_p(q)^{1/2}}{q M_p(q)} \, \frac{\bigl( \sum_{i = 1}^n \sigma_{n, i}^{-2 q} \bigr)^{1/2}}{\sum_{i = 1}^n \sigma_{n, i}^{-q}} \, \widetilde{\xi}_n\\
&\qquad - \frac{\widetilde{\eta}_n}{\sqrt{n p}} + \BigO\biggl( \frac{\sum_{i = 1}^n \sigma_{n, i}^{-2 q}}{\bigl( \sum_{i = 1}^n \sigma_{n, i}^{-q} \bigr)^2} + \frac{1}{n} \biggr) \biggr),
\end{align*}
or equivalently,
\begin{multline}\label{eq:var2_gl}
\frac{\sum_{i = 1}^n \sigma_{n, i}^{-q}}{\bigl( \sum_{i = 1}^n \sigma_{n, i}^{-2 q} \bigr)^{1/2}} \biggl( \frac{n^{1/p} \, \lVert \inv{\Sigma_n} Z_n \rVert_q}{M_p(q)^{1/q} \, \bigl( \sum_{i = 1}^n \sigma_{n, i}^{-q} \bigr)^{1/q}} - 1 \biggr) \stackrel{\text{\upshape d}}{=}\\
\stackrel{\text{\upshape d}}{=} \frac{V_p(q)^{1/2}}{q M_p(q)} \, \widetilde{\xi}_n - \frac{\sum_{i = 1}^n \sigma_{n, i}^{-q}}{\sqrt{n} \bigl( \sum_{i = 1}^n \sigma_{n, i}^{-2 q} \bigr)^{1/2}} \, \frac{\widetilde{\eta}_n}{\sqrt{p}} + \BigO\biggl( \frac{\bigl( \sum_{i = 1}^n \sigma_{n, i}^{-2 q} \bigr)^{1/2}}{\sum_{i = 1}^n \sigma_{n, i}^{-q}} + \frac{\sum_{i = 1}^n \sigma_{n, i}^{-q}}{n \bigl( \sum_{i = 1}^n \sigma_{n, i}^{-2 q} \bigr)^{1/2}} \biggr).
\end{multline}
From the definition of \(G_{q, \bsigma}\), we obtain
\begin{equation*}
\frac{\sum_{i = 1}^n \sigma_{n, i}^{-q}}{n \bigl( \sum_{i = 1}^n \sigma_{n, i}^{-2 q} \bigr)^{1/2}} = \frac{1}{\sqrt{n}} \, \frac{\sum_{i = 1}^n \sigma_{n, i}^{-q}}{\sqrt{n} \bigl( \sum_{i = 1}^n \sigma_{n, i}^{-2 q} \bigr)^{1/2}} \xrightarrow[n \to \infty]{} 0 \cdot G_{q, \bsigma} = 0.
\end{equation*}
Therefore, the right-hand side of~\eqref{eq:var2_gl} converges almost surely, and hence also in distribution, towards
\begin{equation*}
\frac{V_p(q)^{1/2}}{q M_p(q)} \, \widetilde{\xi} - \frac{G_{q, \bsigma}}{\sqrt{p}} \, \widetilde{\eta} \sim \Norm(0, s^2),
\end{equation*}
where
\begin{align*}
s^2 &= \frac{V_p(q)}{q^2 M_p(q)^2} \Var[\widetilde{\xi}] - \frac{2 V_p(q)^{1/2} G_{q, \bsigma}}{q \sqrt{p} M_p(q)} \Cov[\widetilde{\xi}, \widetilde{\eta}] + \frac{G_{q, \bsigma}^2}{p} \Var[\widetilde{\eta}]\\
&= \frac{V_p(q)}{q^2 M_p(q)^2} - \frac{2 V_p(q)^{1/2} G_{q, \bsigma}}{q \sqrt{p} M_p(q)} \, \frac{C_p(p, q) G_{q, \bsigma}}{\sqrt{p \, V_p(q)}} + \frac{G_{q, \bsigma}^2}{p}\\
&= \frac{V_p(q)}{q^2 M_p(q)^2} - \frac{2 C_p(p, q) G_{q, \bsigma}^2}{p q M_p(q)} + \frac{G_{q, \bsigma}^2}{p}.
\end{align*}
From~\eqref{eq:var2_gl} also follows convergence in distribution of the left-hand side, which is exactly what we have claimed.

\textit{Case~\(p = \infty\):} Here \(Z_n = \vek{X}_n\) and so we have
\begin{align*}
\lVert \inv{\Sigma_n} Z_n \rVert_q &= \biggl( \sum_{i = 1}^n \Bigl\lvert \frac{X_i}{\sigma_{n, i}} \Bigr\rvert^q \biggr)^{1/q}\\
&= \biggl( V_\infty(q)^{1/2} \biggl( \sum_{i = 1}^n \sigma_{n, i}^{-2 q} \biggr)^{1/2} \, \xi_n + M_\infty(q) \sum_{i = 1}^n \sigma_{n, i}^{-q} \biggr)^{1/q}\\
&= M_\infty(q)^{1/q} \biggl( \sum_{i = 1}^n \sigma_{n, i}^{-q} \biggr)^{1/q} \biggl( 1 + \frac{V_\infty(q)^{1/2} \bigl( \sum_{i = 1}^n \sigma_{n, i}^{-2 q} \bigr)^{1/2}}{M_\infty(q) \sum_{i = 1}^n \sigma_{n, i}^{-q}} \, \xi_n \biggr)^{1/q}\\
&= M_\infty(q)^{1/q} \biggl( \sum_{i = 1}^n \sigma_{n, i}^{-q} \biggr)^{1/q} \biggl( 1 + \frac{V_\infty(q)^{1/2} \bigl( \sum_{i = 1}^n \sigma_{n, i}^{-2 q} \bigr)^{1/2}}{q M_\infty(q) \sum_{i = 1}^n \sigma_{n, i}^{-q}} \, \xi_n + \BigO\biggl( \frac{\sum_{i = 1}^n \sigma_{n, i}^{-2 q}}{\bigl( \sum_{i = 1}^n \sigma_{n, i}^{-q} \bigr)^2} \biggr) \biggr),
\end{align*}
where we have used Taylor's expansion of \(x \mapsto F(x, 0)\) with the same \(F\) as in the previous case; equivalently,
\begin{equation*}
\frac{\sum_{i = 1}^n \sigma_{n, i}^{-q}}{\bigl( \sum_{i = 1}^n \sigma_{n, i}^{-2 q} \bigr)^{1/2}} \biggl( \frac{\lVert \inv{\Sigma_n} Z_n \rVert_q}{M_\infty(q)^{1/q} \bigl( \sum_{i = 1}^n \sigma_{n, i}^{-q} \bigr)^{1/q}} - 1 \biggr) = \frac{V_\infty(q)^{1/2}}{q M_\infty(q)} \, \xi_n + \BigO\biggl( \frac{\bigl( \sum_{i = 1}^n \sigma_{n, i}^{-2 q} \bigr)^{1/2}}{\sum_{i = 1}^n \sigma_{n, i}^{-q}} \biggr).
\end{equation*}
Lemma~\ref{lem:xi_eta_clt}  says that \((\xi_n)_{n \in \N} \xrightarrow{\mathrm{d}} \xi \sim \Norm(0,1 )\) here. By applying the Skorokhod-representation to \((\xi_n)_{n \in \N}\) and \(\xi\), and observing the conditions on \(\bsigma\) as before, we obtain convergence in distribution towards \(\Norm\bigl( 0, \frac{V_\infty(q)}{q^2 M_\infty(q)^2} \bigr)\). Note that by setting \(\frac{1}{\infty} = 0\) and respecting our convention \(C_\infty(\infty, q) = 0\) the formula for \(s^2\) given in Theorem~\ref{thm:clt_ellipsoid} also is valid for \(p = \infty\) as just proved.
\end{proof}


The proof of Corollary~\ref{cor:threshold_theorem_crit} is now a straightforward application of Theorem~\ref{thm:clt_ellipsoid}.

\begin{proof}[Proof of Corollary~\ref{cor:threshold_theorem_crit}]
Beside the actual statement we are also going to prove that the results of Theorem~\ref{thm:threshold_theorem} are reproduced by the CLT. So let \(t \in (0, \infty)\) and \(n \in \N\), then we can write
\begin{align*}
\vol_n(\D_p^n \cap t \nEll_{\sigma_n}^n) &= \vol_n\bigl( r_{n, p} \B_p^n \cap t r_{n, q} \det(\Sigma_n)^{-\frac{1}{n}} \, \Sigma_n \B_q^n \bigr) = \Pro\Bigl[Z_n \in \frac{t r_{n, q}}{r_{n, p}} \det(\Sigma_n)^{-\frac{1}{n}} \, \Sigma_n \B_q^n \Bigr]\\
&= \Pro\Bigl[ \lVert \inv{\Sigma_n} Z_n \rVert_q \leq \frac{t r_{n, q}}{r_{n, p}} \det(\Sigma_n)^{-\frac{1}{n}} \Bigr]\\
&= \Pro\Biggl[ \frac{\sum_{i = 1}^n \sigma_{n, i}^{-q}}{\bigl( \sum_{i = 1}^n \sigma_{n, i}^{-2 q} \bigr)^{1/2}} \biggl( \frac{n^{1/p} \, \lVert \inv{\Sigma_n} Z_n \rVert_q}{M_p(q)^{1/q} \bigl( \sum_{i = 1}^n \sigma_{n, i}^{-q} \bigr)^{1/q}} - 1 \biggr) \\
&\qquad\quad \leq \frac{\sum_{i = 1}^n \sigma_{n, i}^{-q}}{\bigl( \sum_{i = 1}^n \sigma_{n, i}^{-2 q} \bigr)^{1/2}} \biggl( \frac{t n^{1/p} \, r_{n, q} \det(\Sigma_n)^{-1/n}}{r_{n, p} M_p(q)^{1/q} \bigl( \sum_{i = 1}^n \sigma_{n, i}^{-q} \bigr)^{1/q}} - 1 \biggr) \Biggr].
\end{align*}
Observe that on the right-hand side, we have
\begin{equation*}
\frac{\sum_{i = 1}^n \sigma_{n, i}^{-q}}{\bigl( \sum_{i = 1}^n \sigma_{n, i}^{-2 q} \bigr)^{1/2}} \xrightarrow[n \to \infty]{} \infty \quad \text{and} \quad \frac{n^{1/p} \, r_{n, q} \det(\Sigma_n)^{-1/n}}{r_{n, p} M_p(q)^{1/q} \bigl( \sum_{i = 1}^n \sigma_{n, i}^{-q} \bigr)^{1/q}} \xrightarrow[n \to \infty]{} \frac{A_{p, q}}{F_{q, \bsigma}},
\end{equation*}
and so, if \(t A_{p, q} < F_{q, \bsigma}\), then the term in parentheses converges towards a negative number and the volume in question converges towards zero, and simlarly if \(t A_{p, q} > F_{q, \bsigma}\), then the volume converges towards one. It remains to examine the critical case \(t A_{p, q} = F_{q, \bsigma}\). Closer inspection of the convergence towards \(A_{p, q}\) (for which see~\cite{KPT2019_I}, proof of Corollary~2.1) reveals
\begin{equation*}
\frac{n^{1/p} \, r_{n, q}}{n^{1/q} \, r_{n, p} M_p(q)^{1/q}} = A_{p, q} \Bigl( 1 + \BigO\Bigl( \frac{1}{n} \Bigr) \Bigr),
\end{equation*}
thus we can rewrite
\begin{align*}
\frac{F_{q, \bsigma} n^{1/p} \, r_{n, q} \det(\Sigma_n)^{-1/n}}{A_{p, q} r_{n, p} M_p(q)^{1/q} \bigl( \sum_{i = 1}^n \sigma_{n, i}^{-q} \bigr)^{1/q}} - 1 &= \Bigl( 1 + \BigO\Bigl( \frac{1}{n} \Bigr) \Bigr) \frac{1}{h_n} - 1\\
&= \BigO\Bigl( \frac{1}{n} \Bigr) - \frac{h_n - 1}{h_n}.
\end{align*}
By the comparison of the \(\ell_1\)- and \(\ell_2\)-norms, we see
\begin{equation*}
1 \leq \frac{\sum_{i = 1}^n \sigma_{n, i}^{-q}}{\bigl( \sum_{i = 1}^n \sigma_{n, i}^{-2 q} \bigr)^{1/2}} \leq \sqrt{n},
\end{equation*}
and together with \(\lim_{n \to \infty} h_n = 1\) there follows
\begin{equation*}
\frac{\sum_{i = 1}^n \sigma_{n, i}^{-q}}{\bigl( \sum_{i = 1}^n \sigma_{n, i}^{-2 q} \bigr)^{1/2}} \biggl( \frac{F_{q, \bsigma} n^{1/p} \, r_{n, q} \det(\Sigma_n)^{-1/n}}{A_{p, q} r_{n, p} M_p(q)^{1/q} \bigl( \sum_{i = 1}^n \sigma_{n, i}^{-q} \bigr)^{1/q}} - 1 \biggr) \xrightarrow[n \to \infty]{} -z.
\end{equation*}
This leads to the desired result.
\end{proof}

\section{Proof of the examples}
\label{sec:proof_spec_sigma}


Part~(a) of Proposition~\ref{prop:spec_sigma} shall be tackled directly. In order to prove Proposition~\ref{prop:spec_sigma},~(b), we need a few elementary preliminary lemmas, all of which concern the asymptotics of the terms involved. Their proofs are included for the sake of completeness.

\begin{lemma}\label{lem:factorial_power}
For any \(\alpha \in \R\), as \(n \to \infty\),
\begin{equation*}
n!^{\alpha/n} = \frac{n^\alpha}{\ee^\alpha} \Bigl( 1 + \frac{\alpha \log(n)}{2 n} + \BigO\Bigl( \frac{1}{n} \Bigr) \Bigr) = \frac{n^\alpha}{\ee^\alpha} \Bigl( 1 + \BigO\Bigl( \frac{\log(n)}{n} \Bigr) \Bigr).
\end{equation*}
\end{lemma}

\begin{proof}
This is little more than an application of Stirling's formula and Taylor's expansion of the exponential function: we have
\begin{align*}
n!^{\alpha/n} &= \Bigl( \sqrt{2 \pi n} \Bigl( \frac{n}{\ee} \Bigr)^n \, \ee^{R_n} \Bigr)^{\frac{\alpha}{n}}\\
&= \frac{n^\alpha}{\ee^\alpha} \, \ee^{\frac{\alpha \log(2 \pi n)}{2 n} + \frac{\alpha R_n}{n}}\\
&= \frac{n^\alpha}{\ee^\alpha} \Bigl( 1 + \frac{\alpha \log(2 \pi n)}{2 n} + \frac{\alpha R_n}{n} + \BigO\Bigl( \Bigl[ \frac{\log(2 \pi n)}{2 n} + \frac{R_n}{n} \Bigr]^2 \Bigr) \Bigr).
\end{align*}
The claim follows by observing \(R_n = \BigO(\frac{1}{n})\) and \(\bigl( \frac{R_n}{n} + \frac{\log(2 \pi n)}{2 n} \bigr)^2 = \smallO(\frac{1}{n})\).
\end{proof}

\begin{lemma}\label{lem:sum_powers}
Let \(\alpha, \beta \in \R\). Then, as \(n \to \infty\),
\begin{equation*}
\sum_{i = 1}^n i^\alpha \log(i + 1)^\beta =
\begin{cases} \BigTh(1) &: \alpha < -1 \text{ or } \alpha = -1 \wedge \beta < -1,\\
\log(\log(n + 1)) \bigl( 1 + \BigTh\bigl( \frac{1}{\log(\log(n + 1))} \bigr) \bigr)&: \alpha = \beta = -1,\\
\frac{\log(n + 1)^{\beta + 1}}{\beta + 1} \bigl( 1 + \BigO\bigl( \frac{1}{\log(n + 1)^{\beta + 1}} \bigr) \bigr)&: \alpha = -1 \wedge \beta > -1,\\
\frac{n^{\alpha + 1}}{\alpha + 1} \log(n + 1)^\beta \bigl( 1 - \frac{\beta (1 + \smallO(1))}{(\alpha + 1) \log(n + 1)} \bigr) &: \alpha > -1 \wedge \beta \neq 0,\\
\frac{n^{\alpha + 1}}{\alpha + 1} \bigl( 1 + \BigTh\bigl( \frac{1}{n^{\alpha + 1}} \bigr) \bigr) &: \alpha \in (-1, 0) \wedge \beta = 0,\\
\frac{n^{\alpha + 1}}{\alpha + 1} \bigl( 1 + \BigTh\bigl( \frac{1}{n} \bigr) \bigr) &: \alpha \geq 0 \wedge \beta = 0.
\end{cases}
\end{equation*}
(Here \(\BigTh(b_n)\) is intended to denote \emph{sharp} asymptotics, that is, it should be read \(b_n (c + \smallO(1))\) with some \(c \neq 0\).)
\end{lemma}

\begin{proof}
\textit{Case~\(\alpha < -1\):} This is clear because the series converges.

For the remaining (mostly diverging) cases we recall the Euler\--Maclaurin formula: for any differentiable function \(f \colon [1, \infty) \to \R\) with \(f'\) Riemann\-/integrable over compact intervals, and for any \(n \in \N\),
\begin{equation}\label{eq:eulermaclaurin}
\sum_{k = 1}^n f(k) = \int_1^n f(x) \, \dx x + \frac{f(1) + f(n)}{2} + \int_1^n H(x) f'(x) \, \dx x,
\end{equation}
where \(H \colon \R \to \R\) is defined by \(H(x) := x - \lfloor x \rfloor - \frac{1}{2}\); note \(\lvert H(x) \rvert \leq \frac{1}{2}\) for any \(x \in \R\). In the sequel we take \(f(x) := x^\alpha \log(x + 1)^\beta\). Let \(n \in \N\).

\textit{Case \(\alpha = -1 \wedge \beta < -1\):} Observe \(\frac{1}{x} = \frac{1}{x + 1} + \frac{1}{x (x + 1)}\), and the latter terms satisfies \(\frac{1}{2 x^2} \leq \frac{1}{x (x + 1)} \leq \frac{1}{x^2}\) for all \(x \geq 1\). This yields
\begin{equation}\label{eq:alpha-1betan-1}
\int_1^n \frac{\log(x + 1)^\beta}{x} \, \dx x = \frac{\log(n + 1)^{\beta + 1} - \log(2)^{\beta + 1}}{\beta + 1} + \int_1^n \frac{\log(x + 1)^\beta}{x (x + 1)} \, \dx x,
\end{equation}
where the second integral converges as \(n \to \infty\). Next, \(f\) is decreasing, therefore the middle term in~\eqref{eq:eulermaclaurin} converges; and \(f' \leq 0\), therefore
\begin{equation*}
\biggl\lvert \int_1^n H(x) f'(x) \, \dx x \biggr\rvert \leq -\frac{1}{2} \int_1^n f'(x) \, \dx x = \frac{f(1) - f(n)}{2},
\end{equation*}
wherefore also the third term converges. Because of \(\beta + 1 < 0\) the right\-/hand side of~\eqref{eq:eulermaclaurin} converges, and thus so does the left\-/hand side.

\textit{Case \(\alpha = \beta = -1\):} The techniques are the same as before, but now
\begin{equation*}
\int_1^n \frac{1}{x \log(x + 1)} \, \dx x = \log(\log(n + 1)) - \log(\log(2)) + \int_1^n \frac{1}{x (x + 1) \log(x + 1)} \, \dx x,
\end{equation*}
where the second integral converges still, so \(\log(\log(n + 1))\) dominates. The second and third terms of~\eqref{eq:eulermaclaurin} also converge, and this results in
\begin{equation*}
\sum_{i = 1}^n \frac{1}{i \log(i + 1)} = \log(\log(n + 1)) + c + \smallO(1),
\end{equation*}
with a certain \(c \in \R\); to be precise, \(c > 0\) holds true for the following argument, exploiting monotonicity throughout,
\begin{align*}
\sum_{i = 1}^n \frac{1}{i \log(i + 1)} &\geq \int_1^{n + 1} \frac{1}{x \log(x + 1)} \, \dx x \geq \int_1^{n + 1} \frac{1}{(x + 1) \log(x + 1)} \, \dx x\\
&= \log(\log(n + 2)) - \log(\log(2)) \geq \log(\log(n + 1)) - \log(\log(2)),
\end{align*}
hence \(c = \lim_{n \to \infty} \bigl( \sum_{i = 1}^n \frac{1}{i \log(i + 1)} - \log(\log(n + 1)) \bigr) \geq -\log(\log(2)) > 0\). This implies the claim.

\textit{Case \(\alpha = -1 \wedge \beta > -1\):} The second and third terms of~\eqref{eq:eulermaclaurin} still converge; the first term, that is, the integral, is the same as in~\eqref{eq:alpha-1betan-1}, where again the second integral converges, but now \(\log(n + 1)^{\beta + 1}\) is unbounded and therefore dominates the scene, hence
\begin{equation*}
\sum_{i = 1}^n \frac{\log(i + 1)^\beta}{i} = \frac{\log(n + 1)^{\beta + 1}}{\beta + 1} + c + \smallO(1),
\end{equation*}
again with a certain \(c \in \R\) (whose value in this case cannot be estimated so easily), and this case is accounted for. (That much may be said about \(c\): For \(\beta \leq \log(4)\), \(f\) is decreasing, and then \(c \geq \int_1^\infty \frac{\log(x + 1)^{\beta + 1}}{(\beta + 1) x^2} \, \dx x > 0\) can be proved. For \(\beta > \log(4)\), \(f\) is only eventually decreasing.)

\textit{Case \(\alpha > -1 \wedge \beta \neq 0\):} Depending on the values of \(\alpha\) and \(\beta\) the second and third terms of~\eqref{eq:eulermaclaurin} either both remain bounded or both are of order \(\BigO(n^\alpha \log(n + 1)^\beta)\). The first term is tackled with partial integration as follows,
\begin{equation*}
\int_1^n x^\alpha \log(x + 1)^\beta \, \dx x = \frac{n^{\alpha + 1} \log(n + 1)^\beta - \log(2)^\beta}{\alpha + 1} - \frac{\beta}{\alpha + 1} \int_1^n \frac{x^{\alpha + 1}}{x + 1} \log(x + 1)^{\beta - 1} \, \dx x.
\end{equation*}
Now note \(\frac{1}{2} \leq \frac{x}{x + 1} \leq 1\), therefore the last integrand behaves like \(x^\alpha \log(x + 1)^{\beta - 1}\). Using L'Hospital's rule we recognize
\begin{equation}\label{eq:int_potenz_log}
\int_1^n x^\alpha \log(x + 1)^\beta \, \dx x = \frac{x^{\alpha + 1} \log(x + 1)^\beta}{\alpha + 1} \Bigl( 1 - \frac{\beta (1 + \smallO(1))}{(\alpha + 1) \log(n + 1)} \Bigr);
\end{equation}
therefore by now we have
\begin{align*}
\sum_{i = 1}^n i^\alpha \log(i + 1)^\beta &= \frac{x^{\alpha + 1} \log(x + 1)^\beta}{\alpha + 1} \Bigl( 1 - \frac{\beta (1 + \smallO(1))}{(\alpha + 1) \log(n + 1)} \Bigr)\\
&\quad + \frac{n^\alpha \log(n + 1)^\beta + \log(2)^\beta}{2} + \int_1^n H(x) f'(x) \, \dx x.
\end{align*}
In the case \(\alpha < 0\) or \(\alpha = 0 \wedge \beta < 0\), \(f\) is eventually decreasing and converging to zero; hence, with \(1 \leq M \leq n\) sufficiently large, we have
\begin{equation*}
\int_M^n \bigl\lvert H(x) f'(x) \bigr\rvert \, \dx x \leq -\frac{1}{2} \int_M^n f'(x) \, \dx x = \frac{f(M) - f(n)}{2} \leq \frac{f(M)}{2},
\end{equation*}
so \(\int_1^\infty H(x) f'(x) \, \dx x\) converges absolutely, and we have, again for \(n\) sufficiently large,
\begin{equation*}
\biggl\lvert \int_1^\infty H(x) f'(x) \, \dx x - \int_1^n H(x) f'(x) \, \dx x \biggr\rvert \leq \int \int_n^\infty \bigl\lvert H(x) f'(x) \bigr\rvert \, \dx x \leq \frac{f(n)}{2},
\end{equation*}
that is, \(\int_1^n H(x) f'(x) \, \dx x = c + \BigO(f(n))\) with some constant \(c \in \R\). In the case \(\alpha = 0 \wedge \beta > 0\) or \(\alpha > 0\), \(f\) is eventually incrasing and unbounded; so let again \(1 \leq M \leq n\) sufficiently large, then
\begin{equation*}
\biggl\lvert \int_1^n H(x) f'(x) \, \dx x \biggr\rvert \leq \biggl\lvert \int_1^M H(x) f'(x) \, \dx x \biggr\rvert + \frac{f(n) - f(M)}{2},
\end{equation*}
and therefore \(\int_1^n H(x) f'(x) \, \dx x = \BigO(f(n))\). So in either case we get, with some suitable constant \(c \in \R\),
\begin{align*}
\sum_{i = 1}^n i^\alpha \log(i + 1)^\beta &= \frac{n^{\alpha + 1} \log(n + 1)^\beta}{\alpha + 1} \Bigl( 1 - \frac{\beta (1 + \smallO(1))}{(\alpha + 1) \log(n + 1)} \Bigr) + c + \BigO(n^\alpha \log(n + 1)^\beta)\\
&= \frac{n^{\alpha + 1} \log(n + 1)^\beta}{\alpha + 1} \biggl( 1 - \frac{\beta (1 + \smallO(1))}{(\alpha + 1) \log(n + 1)} + \frac{c (\alpha + 1)}{n^{\alpha + 1} \log(n + 1)^\beta} + \BigO\Bigl( \frac{1}{n} \Bigr) \biggr).
\end{align*}
Therefore the error term is dominated by \(\frac{-\beta (1 + \smallO(1))}{(\alpha + 1) \log(n + 1)}\).

\textit{Case \(\alpha > -1 \wedge \beta = 0\):} This is a consequence of McGown and Parks~\cite{McGP2007} who provide the exact asymptotics for diverging sums of non\-/integral powers. The proof is complete.
\end{proof}

The next lemma extends upon the statement for the case \(\alpha = 0 \wedge \beta \neq 0\). Obviously it could be done for any \(\alpha > -1\); we restrict ourselves to said case solely because it is the only one we need, and we do not wish to overburden the present article.

\begin{lemma}\label{lem:alpha0betan0}
Let \(\beta \in \R \setminus \{0\}\) and let \(N \in \N\), then the following holds true, as \(n \to \infty\),
\begin{equation*}
\sum_{i = 1}^n \log(i + 1)^\beta = n \log(n + 1)^\beta \biggl( \sum_{k = 0}^{N - 1} (-1)^k \frac{(\beta)_k}{\log(n + 1)^k} + (-1)^N \frac{(\beta)_N (1 + \smallO(1))}{\log(n + 1)^N} \biggr),
\end{equation*}
where \((\beta)_k := \prod_{l = 0}^{k - 1} (\beta - l)\) is the falling factorial.
\end{lemma}

\begin{proof}
As before we use the Euler\--Maclaurin formula,
\begin{equation*}
\sum_{i = 1}^n \log(i + 1)^\beta = \int_1^n \log(x + 1)^\beta \, \dx x + \frac{\log(n + 1)^\beta + \log(2)^\beta}{2} + \int_1^n H(x) \frac{\beta \log(x + 1)^{\beta - 1}}{x + 1} \, \dx x,
\end{equation*}
where \(H\) is the same as in the proof of Lemma~\ref{lem:sum_powers}. We see immediately that the last integral behaves like \(c + \BigO(\log(n + 1)^\beta)\), that is, it is comparabale to the middle term. It remains to evaluate the first integral; repeated partial integration reveals
\begin{equation*}
\int_1^n \log(x + 1)^\beta \, \dx x = \sum_{k = 0}^{N - 1} (-1)^k (\beta)_k (n + 1) \log(n + 1)^{\beta - k} + c + (-1)^N (\beta)_N \int_1^n \log(x + 1)^{\beta - N} \, \dx x,
\end{equation*}
where we have gathered all constants in the symbol \(c\). From~\eqref{eq:int_potenz_log} we know already that
\begin{equation*}
\int_1^n \log(x + 1)^{\beta - N} = n \log(n + 1)^{\beta - N} (1 + \smallO(1)),
\end{equation*}
and therewith we continue
\begin{align*}
\int_1^n \log(x + 1)^\beta \, \dx x &= (n + 1) \log(n + 1) \biggl( \sum_{k = 0}^{N - 1} (-1)^k \frac{(\beta)_k}{\log(n + 1)^k}\\
&\quad + \frac{c}{(n + 1) \log(n + 1)^\beta} + (-1)^N \frac{n (\beta)_N (1 + \smallO(1))}{(n + 1) \log(n + 1)^N} \biggr),
\end{align*}
and since \(\frac{n}{n + 1} = 1 - \frac{1}{n + 1}\) we can absorb this term into the already existing \(1 + \smallO(1)\); next, the term \(\frac{c}{(n + 1) \log(n + 1)^\beta}\) is negligible compared to the other, purely logarithmic terms; also \(n + 1 = n (1 + \frac{1}{n})\), and as the asymptotic part is of order \(1 - \frac{\beta (1 + \smallO(1))}{\log(n + 1)}\), we can safely replace the leading factor \(n + 1\) by \(n\). In total we get
\begin{align*}
\sum_{i = 1}^n \log(i + 1)^\beta &= n \log(n + 1) \biggl( \sum_{k = 0}^{N - 1} (-1)^k \frac{(\beta)_k}{\log(n + 1)^k} + (-1)^N \frac{(\beta)_N (1 + \smallO(1))}{\log(n + 1)^N}\\
&\mspace{130mu} + \frac{c}{n \log(n + 1)^\beta} + \BigO\Bigl( \frac{1}{n} \Bigr) \biggr),
\end{align*}
and the last two terms are clearly immaterial.
\end{proof}

\begin{lemma}\label{lem:prod_log}
The following holds true as \(n \to \infty\),
\begin{equation*}
\biggl( \prod_{i = 1}^n \log(i + 1) \biggr)^{1/n} = \log(n + 1) \Bigl( 1 - \frac{1}{\log(n + 1)} - \frac{1}{2 \log(n + 1)^2} - \frac{7}{6 \log(n + 1)^3} - \frac{95 + \smallO(1)}{24 \log(n + 1)^4} \Bigr).
\end{equation*}
\end{lemma}

\begin{proof}
First we take the logarithm, so we can work with \(\frac{1}{n} \sum_{i = 1}^n \log(\log(i + 1))\). Then we are going to employ the Euler\--Maclaurin formula,
\begin{align*}
\sum_{i = 1}^n \log(\log(i + 1)) &= \int_1^n \log(\log(x + 1)) \, \dx x\\
&\quad + \frac{\log(\log(n + 1)) + \log(\log(2))}{2} + \int_1^n \frac{H(x)}{(x + 1) \log(x + 1)} \, \dx x,
\end{align*}
where \(H\) is the same as in the proof of Lemma~\ref{lem:sum_powers}. By repeated partial integration, the first integral can be seen to be
\begin{align*}
\int_1^n \log(\log(x + 1)) \, \dx x &= (n + 1) \log(\log(n + 1)) - 2 \log(\log(2))\\
&\quad - \frac{n + 1}{\log(n + 1)} + \frac{2}{\log(2)} - \frac{n + 1}{\log(n + 1)^2} + \frac{2}{\log(2)^2}\\
&\quad - 2 \int_1^n \frac{1}{\log(x + 1)^3} \, \dx x,
\end{align*}
and from Equation~\eqref{eq:int_potenz_log} we already know
\begin{equation*}
\int_1^n \frac{1}{\log(x + 1)^3} \, \dx x = \frac{n}{\log(n + 1)^3} \Bigl( 1 + \frac{3 + \smallO(1)}{\log(n + 1)} \Bigr).
\end{equation*}
The usual estimate also yields
\begin{equation*}
\biggl\lvert \int_1^n \frac{H(x)}{(x + 1) \log(x + 1)} \, \dx x \biggr\rvert \leq \frac{\log(\log(n + 1)) - \log(\log(2))}{2}.
\end{equation*}
Putting things together shows us
\begin{align*}
\frac{1}{n} \sum_{i = 1}^n \log(\log(n + 1)) &= \log(\log(n + 1)) - \frac{1}{\log(n + 1)} - \frac{1}{\log(n + 1)^2}\\
&\quad - \frac{2}{\log(n + 1)^3} - \frac{6 + \smallO(1)}{\log(n + 1)^4} + \BigO\Bigl( \frac{\log(\log(n + 1))}{n} \Bigr),
\end{align*}
and obviously the last big\-/O term even may be neglected. Taking the exponential again on both sides and using the Taylor series of the exponential function on the right\-/hand side leads to the claimed result.
\end{proof}

\begin{proof}[Proof of Proposition~\ref{prop:spec_sigma}]
(a)~Using division with remainder, write \(n - n_0 = m_n d + r_n\) with \(m_n, r_n \in \N_0\) and \(r_n \leq d - 1\), for each \(n \in \N\) with \(n \geq n_0 + 1\). Then \(\lim_{n \to \infty} m_n = \infty\) and \(\lim_{n \to \infty} \frac{n}{m_n} = d\) hold. These imply
\begin{equation*}
\biggl( \prod_{i = 1}^n \sigma_i \biggr)^{1/n} = \biggl( \prod_{i = 1}^{n_0} \sigma_i \biggr)^{1/n} \biggl( \prod_{i = 1}^{d} \sigma_{n_0 + i} \biggr)^{m_n/n} \biggl( \prod_{i = 1}^{r_n} \sigma_{n_0 + i} \biggr)^{1/n} \xrightarrow[n \to \infty]{} \biggl( \prod_{i = 1}^{d} \sigma_{n_0 + i} \biggr)^{1/d}
\end{equation*}
and analogously
\begin{equation*}
\frac{1}{n} \sum_{i = 1}^n \sigma_i^{\beta} = \frac{1}{n} \sum_{i = 1}^{n_0} \sigma_i^\beta + \frac{m_n}{n} \sum_{i = 1}^{d} \sigma_{n_0 + i}^\beta + \frac{1}{n} \sum_{i = 1}^{r_n} \sigma_{n_0 + i}^\beta \xrightarrow[n \to \infty]{} \frac{1}{d} \sum_{i = 1}^d \sigma_{n_0 + i}^\beta
\end{equation*}
for any \(\beta \in \R\). From these results it follows that
\begin{equation*}
F_{q, \bsigma} = \biggl( \prod_{i = 1}^{d} \sigma_{n_0 + i} \biggr)^{1/d} \biggl( \frac{1}{d} \sum_{i = 1}^{d} \sigma_{n_0 + i}^{-q} \biggr)^{1/q},
\end{equation*}
as well as (referring to Remark~\ref{rmk:gqsigma})
\begin{equation*}
G_{q, \bsigma} = \frac{\sum_{i = 1}^{d} \sigma_{n_0 + i}^{-q}}{\sqrt{d} \bigl( \sum_{i = 1}^{d} \sigma_{n_0 + i}^{-2 q} \bigr)^{1/2}},
\end{equation*}
and (w.l.o.g.\ \(n \geq n_0 + d\))
\begin{equation*}
\frac{\max\limits_{1 \leq i \leq n} \sigma_i^{-2 q}}{\sum_{i = 1}^n \sigma_i^{-2 q}} = \frac{\max\limits_{1 \leq i \leq n} \sigma_{n_0 + i}^{-2 q}}{\sum_{i = 1}^{n_0} \sigma_i^{-2 q} + m_n \sum_{i = 1}^{d} \sigma_{n_0 + i}^{-2 q} + \sum_{i = 1}^{r_n} \sigma_{n_0 + i}^{-2 q}} \xrightarrow[n \to \infty]{} 0.
\end{equation*}
Recall the definition of \(h_n\) and \(z\) in Corollary~\ref{cor:threshold_theorem_crit}; concerning \(z\), we obtain
\begin{align*}
\frac{\sum_{i = 1}^n \sigma_i^{-q}}{\bigl( \sum_{i = 1}^n \sigma_i^{-2 q} \bigr)^{1/2}} (h_n - 1) &= \sqrt{n} \, \frac{\frac{1}{n} \sum_{i = 1}^{n_0} \sigma_i^{-q} + \frac{m_n}{n} \sum_{i = 1}^d \sigma_{n_0 + i}^{-q} + \frac{1}{n} \sum_{i = 1}^{r_n} \sigma_{n_0 + i}^{-q}}{\bigl( \frac{1}{n} \sum_{i = 1}^{n_0} \sigma_i^{-q} + \frac{m_n}{n} \sum_{i = 1}^d \sigma_{n_0 + i}^{-2 q} + \frac{1}{n} \sum_{i = 1}^{r_n} \sigma_{n_0 + i}^{-2 q} \bigr)^{1/2}} \cdot {}\\
&\qquad \cdot \biggl( \frac{\bigl( \prod_{i = 1}^n \sigma_i \bigr)^{1/n} \bigl( \frac{1}{n} \sum_{i = 1}^n \sigma_i^{-q} \bigr)^{1/q}}{\bigl( \prod_{i = 1}^d \sigma_i \bigr)^{1/d} \bigl( \frac{1}{d} \sum_{i = 1}^d \sigma_i^{-q} \bigr)^{1/q}} - 1 \biggr)\\
&= \sqrt{n} \, \frac{\frac{m_n}{n} \sum_{i = 1}^d \sigma_{n_0 + i}^{-q} + \BigO\bigl( \frac{1}{n} \bigr)}{\bigl( \frac{m_n}{n} \sum_{i = 1}^d \sigma_{n_0 + i}^{-2 q} + \BigO\bigl( \frac{1}{n} \bigr) \bigr)^{1/2}} \cdot {}\\
&\qquad \cdot \Biggl( \biggl( \prod_{i = 1}^{n_0} \sigma_i \cdot \biggl( \prod_{i = 1}^d \sigma_{n_0 + i} \biggr)^{-(n_0 + r_n)/d} \cdot \prod_{i = 1}^{r_n} \sigma_{n_0 + i} \biggr)^{1/n} \cdot {}\\
&\qquad \cdot \biggl( \frac{n - n_0 - r_n}{n} + \frac{d}{n} \frac{\sum_{i = 1}^{n_0} \sigma_i^{-q} + \sum_{i = 1}^{r_n} \sigma_{n_0 + i}^{-q}}{\sum_{i = 1}^d \sigma_{n_0 + i}^{-q}} \biggr)^{1/q} - 1 \Biggr)\\
&= \sqrt{n} \, \BigTh(1) \biggl( \Bigl( 1 + \BigO\Bigl( \frac{1}{n} \Bigr) \Bigr) \Bigl( 1 + \BigO\Bigl( \frac{1}{n} \Bigr) \Bigr)^{1/q} - 1 \biggr)\\
&= \sqrt{n} \BigO\Bigl( \frac{1}{n} \Bigr) \xrightarrow[n \to \infty]{} 0.
\end{align*}

(b)~Throughout we have, by Lemmas~\ref{lem:factorial_power} and~\ref{lem:prod_log},
\begin{align*}
\biggl( \prod_{i = 1}^n \sigma_i \biggr)^{1/n} &= \biggl( \prod_{i = 1}^n \bigl( i^\alpha \log(i + 1)^\beta \bigr) \biggr)^{1/n} = n!^{\alpha/n} \log(n + 1)^\beta \Bigl( 1 - \frac{1 + \smallO(1)}{\log(n + 1)} \Bigr)^\beta\\
&= \begin{cases}
\ee^{-\alpha} n^\alpha \bigl( 1 + \frac{\alpha \log(n)}{2 n} (1 + \smallO(1)) \bigr) & \text{if } \beta = 0,\\
\ee^{-\alpha} n^\alpha \log(n + 1)^\beta \bigl( 1 - \frac{\beta}{\log(n + 1)} (1 + \smallO(1)) \bigr) & \text{if } \beta \neq 0.
\end{cases}
\end{align*}
First we are going to check whether the conditions of Theorem~\ref{thm:threshold_theorem} are met.

\textit{Case \(\alpha > \frac{1}{q}\) or \(\alpha = \frac{1}{q} \wedge \beta \geq 0\):} The condition~\eqref{eq:sigma} is violated since
\begin{equation}\label{eq:v_n}
\biggl( \prod_{i = 1}^n \sigma_i \biggr)^{2 q/n} \frac{1}{n^2} \sum_{i = 1}^n \sigma_i^{-2 q} = \frac{n^{2 q \alpha - 2} \log(n + 1)^{2 q \beta}}{\ee^{2 q \alpha}} \BigTh(1),
\end{equation}
which does not converge to zero as \(n \to \infty\).

\textit{Case \(\alpha = \frac{1}{q} \wedge \beta < 0\):} Looking back to~\eqref{eq:v_n}, now~\eqref{eq:sigma} is satisfied, and
\begin{equation*}
\biggl( \prod_{i = 1}^n \sigma_i \biggr)^{q/n} \frac{1}{n} \sum_{i = 1}^n \sigma_i^{-q} = \frac{\log(n + 1)}{\ee (1 - q \beta)} (1 + \smallO(1)),
\end{equation*}
hence \(F_{q, \bsigma} = \infty\).

\textit{Case \(\alpha < \frac{1}{q}\):} From Lemma~\ref{lem:sum_powers} note that
\begin{equation*}
\sum_{i = 1}^n \sigma_i^{-2 q} = \BigO\bigl( (1 + n^{-2 q \alpha + 1}) (1 + \log(n + 1)^{-2 q \beta + 1}) \bigr),
\end{equation*}
so even if the series diverges, it does so of order at most \(n^{-2 q \alpha + 1} \log(n + 1)^{-2 q \beta + 1}\), hence~\eqref{eq:sigma} is fulfilled, and
\begin{align*}
\biggl( \prod_{i = 1}^n \sigma_i \biggr)^{q/n} \frac{1}{n} \sum_{i = 1}^n \sigma_i^{-q} &= \frac{n^{q \alpha - 1} \log(n  + 1)^{q \beta}}{\ee^{q \alpha}} (1 + \smallO(1)) \frac{n^{-q \alpha + 1} \log(n + 1)^{-q \beta}}{-q \alpha + 1} (1 + \smallO(1))\\
&= \frac{1}{\ee^{q \alpha} (1 - q \alpha)} (1 + \smallO(1)),
\end{align*}
which implies the stated value of \(F_{q, \bsigma}\).

\noindent
Now we are going to investigate the premises of Theorem~\ref{thm:clt_ellipsoid} and Corollary~\ref{cor:threshold_theorem_crit}.

\textit{Case \(\alpha > \frac{1}{2 q}\) or \(\alpha = \frac{1}{2 q} \wedge \beta > \frac{1}{2 q}\):} Noether's condition~\eqref{eq:feller_sigma} is not satisfied as
\begin{equation*}
\frac{\max_{i \in [1, n]} \sigma_i^{-2 q}}{\sum_{i = 1}^n \sigma_i^{-2 q}} = \frac{\max_{i \in [1, n]} n^{-2 q \alpha} \log(n + 1)^{-2 q \beta}}{\sum_{i = 1}^n n^{-2 q \alpha} \log(n + 1)^{-2 q \beta}},
\end{equation*}
and both enumerator and denominator are of order \(\BigTh(1)\).

\textit{Case \(\alpha = \frac{1}{2 q} \wedge \beta \leq \frac{1}{2 q}\):} The enumerator in Noether's condition still is of order \(\BigTh(1)\), but the denominator is unbounded, thus~\eqref{eq:feller_sigma} is satisfied. From Remark~\ref{rmk:gqsigma}, (1), we recall that, provided \(F_{q, \bsigma} < \infty\)\---which we have\===, \(G_{q, \bsigma}\) exists iff \(F_{2 q, \bsigma}\) exists, and then \(G_{q, \bsigma} = \frac{F_{q, \bsigma}^q}{F_{2 q, \bsigma}^q}\). For \(\beta = \frac{1}{2 q}\) we get
\begin{equation*}
\biggl( \prod_{i = 1}^n \sigma_i \biggr)^{2 q/n} \frac{1}{n} \sum_{i = 1}^n \sigma_i^{-2 q} = \frac{\log(n + 1)}{\ee} \log(\log(n + 1)) (1 + \smallO(1)),
\end{equation*}
hence \(G_{q, \bsigma} = 0\), and for \(\beta < \frac{1}{2 q}\),
\begin{equation*}
\biggl( \prod_{i = 1}^n \sigma_i \biggr)^{2 q/n} \frac{1}{n} \sum_{i = 1}^n \sigma_i^{-2 q} = \frac{\log(n + 1)^{2 q \beta}}{\ee} \, \frac{\log(n + 1)^{-2 q \beta + 1}}{-2 q \beta + 1} (1 + \smallO(1)) = \frac{\log(n + 1)}{\ee (1 - 2 q \beta)} (1 + \smallO(1)),
\end{equation*}
hence again \(G_{q, \bsigma} = 0\).

The calculations concerning \(z\) require more attention. For \(\beta = \frac{1}{2 q}\) we have
\begin{align*}
h_n &= \ee^{1/(2 q)} \Bigl( \frac{1}{2} \Bigr)^{1/q} \, \frac{n^{1/(2 q)} \log(n + 1)^{1/(2 q)}}{\ee^{1/(2 q)}} \Bigl( 1 - \frac{1 + \smallO(1)}{2 q \log(n + 1)} \Bigr)\\
&\quad \cdot \frac{1}{n^{1/q}} \biggl( \frac{n^{1/2} \log(n + 1)^{-1/2}}{1/2} \Bigl( 1 + \frac{1 + \smallO(1)}{\log(n + 1)} \Bigr) \biggr)^{1/q}\\
&= 1 + \frac{1 + \smallO(1)}{2 q \log(n + 1)},
\end{align*}
and therewith (calling \((z_n)_{n \in \N}\) the defining sequence of \(z\))
\begin{equation*}
z_n = \frac{\frac{n^{1/2} \log(n + 1)^{-1/2}}{1/2} \bigl( 1 + \frac{1 + \smallO(1)}{\log(n + 1)} \bigr)}{\log(\log(n + 1))^{1/2} \bigl( 1 + \BigTh(\frac{1}{\log(\log(n + 1))}) \bigr)^{1/2}} \, \frac{1 + \smallO(1)}{2 q \log(n + 1)} = \frac{n^{1/2} \log(n + 1)^{-3/2}}{q \log(\log(n + 1))^{1/2}} (1 + \smallO(1)),
\end{equation*}
and this converges to infinity. For \(\beta < \frac{1}{2 q}\) with \(\beta \neq 0\) we have
\begin{align*}
h_n &= \ee^{1/(2 q)} \Bigl( \frac{1}{2} \Bigr)^{1/q} \, \frac{n^{1/(2 q)} \log(n + 1)^\beta}{\ee^{1/(2 q)}} \Bigl( 1 - \frac{\beta (1 + \smallO(1))}{\log(n + 1)} \Bigr)\\
&\quad \cdot \frac{1}{n^{1/q}} \biggl( \frac{n^{1/2} \log(n + 1)^{-q \beta}}{1/2} \Bigl( 1 + \frac{2 q \beta (1 + \smallO(1))}{\log(n + 1)} \Bigr) \biggr)^{1/q}\\
&= 1 + \frac{\beta (1 + \smallO(1))}{\log(n + 1)},
\end{align*}
leading to
\begin{equation*}
z_n = \frac{\frac{n^{1/2} \log(n + 1)^{-q \beta}}{1/2} \bigl( 1 + \frac{2 q \beta (1 + \smallO(1))}{\log(n + 1)} \bigr)}{\frac{\log(n + 1)^{1/2 - q \beta}}{(1 - 2 q \beta)^{1/2}} \bigl( 1 + \BigO(\frac{1}{\log(n + 1)^{1 - 2 q \beta}}) \bigr)^{1/2}} \, \frac{\beta (1 + \smallO(1))}{\log(n + 1)} = \frac{2 \beta (1 - 2 q \beta)^{1/2} n^{1/2}}{\log(n + 1)^{3/2}} (1 + \smallO(1)),
\end{equation*}
again converging to infinity if \(\beta > 0\), or to negative infinity if \(\beta < 0\). Lastly for \(\beta = 0\) we have (with some \(c \in \R \setminus\{0\}\))
\begin{align*}
h_n &= \ee^{1/(2 q)} \Bigl( \frac{1}{2} \Bigr)^{1/q} \, \frac{n^{1/(2 q)}}{\ee^{1/(2 q)}} \Bigl( 1 + \frac{\log(n) (1 + \smallO(1))}{4 q n} \Bigr) \frac{1}{n^{1/q}} \biggl( \frac{n^{1/2}}{1/2} \Bigl( 1 + \frac{c (1 + \smallO(1))}{n^{1/2}} \Bigr) \biggr)^{1/q}\\
&= 1 + \frac{c (1 + \smallO(1))}{n^{1/2}},
\end{align*}
and consequently
\begin{align*}
z_n = \frac{\frac{n^{1/2}}{1/2} \bigl( 1 + \frac{c (1 + \smallO(1))}{n^{1/2}} \bigr)}{\log(n + 1)^{1/2} \bigl( 1 + \BigO(\frac{1}{\log(n + 1)}) \bigr)^{1/2}} \, \frac{c (1 + \smallO(1))}{n^{1/2}} = \frac{2 c (1 + \smallO(1))}{\log(n + 1)^{1/2}},
\end{align*}
which yields \(z = 0\).

\textit{Case \(\alpha < \frac{1}{2 q}\):} The denominator for Noether's condition now always reads \(\frac{n^{1 - 2 q \alpha} \log(n + 1)^{-2 q \beta}}{1 - 2 q \alpha} \cdot (1 + \smallO(1))\). For \(\alpha \geq 0\) the enumerator grows at most logarithmically, and hence Noether's condition is fulfilled then. For \(\alpha < 0\), \(n^{-2 q \alpha} \log(n + 1)^{-2 q \beta}\) is eventually increasing and unbounded, therefore the whole defining fraction is of order \(\frac{1}{n}\), and Noether's condition is satisfied.

Next, we know \(F_{2 q, \bsigma} = \frac{1}{\ee^\alpha (1 - 2 q \alpha)^{1/(2 q)}}\) exists, which leads to the claimed value of \(G_{q, \bsigma}\).

In the subcase \(\beta \neq 0\) we get
\begin{align*}
h_n &= \ee^\alpha (1 - q \alpha)^{1/q} \, \frac{n^\alpha \log(n + 1)^\beta}{\ee^\alpha} \Bigl( 1 - \frac{\beta (1 + \smallO(1))}{\log(n + 1)} \Bigr)\\
&\quad \cdot \frac{1}{n^{1/q}} \biggl( \frac{n^{1 - q \alpha} \log(n + 1)^{-q \beta}}{1 - q \alpha} \Bigl( 1 + \frac{q \beta (1 + \smallO(1))}{(1 - q \alpha) \log(n + 1)} \Bigr) \biggr)^{1/q}\\
&= 1 + \frac{q \alpha \beta + \smallO(1)}{(1 - q \alpha) \log(n + 1)},
\end{align*}
and therewith
\begin{align*}
z_n &= \frac{\frac{n^{1 - q \alpha} \log(n + 1)^{-q \beta}}{1 - q \alpha} \bigl( 1 + \frac{q \beta + \small(1)}{(1 - q \alpha) \log(n + 1)} \bigr)}{\frac{n^{1/2 - q \alpha} \log(n + 1)^{-q \beta}}{(1 - 2 q \alpha)^{1/2}} \bigl( 1 + \frac{2 q \beta + \smallO(1)}{(1 - 2 q \alpha) \log(n + 1)} \bigr)^{1/2}} \frac{q \alpha \beta + \smallO(1)}{\log(n + 1)}\\
&= \frac{(1 - 2 q \alpha)^{1/2}}{1 - q \alpha} \, \frac{n^{1/2}}{\log(n + 1)} \bigl( q \alpha \beta + \smallO(1) \bigr),
\end{align*}
which converges to plus or minus infinity, with the same sign as \(\alpha \beta\), provided \(\alpha \neq 0\). The subcase \(\alpha = 0 \wedge \beta \neq 0\) requires the more delicate asymptotics of Lemmas~\ref{lem:alpha0betan0} and~\ref{lem:prod_log}; we have
\begin{align*}
h_n &= \log(n + 1)^\beta \Bigl( 1 - \frac{1}{\log(n + 1)} - \frac{1 + \smallO(1)}{2 \log(n + 1)^2} \Bigr)^\beta\\
&\quad \cdot \frac{1}{n^{1/q}} \biggl( n \log(n + 1)^{-q \beta} \Bigl( 1 + \frac{q \beta}{\log(n + 1)} + \frac{q \beta (q \beta + 1) (1 + \smallO(1))}{\log(n + 1)^2} \Bigr) \biggr)^{1/q}\\
&= 1 + \frac{q \beta^2 (1 + \smallO(1))}{2 \log(n + 1)^2},
\end{align*}
and thus
\begin{equation*}
z_n = \frac{n \log(n + 1)^{-q \beta} (1 + \smallO(1))}{n^{1/2} \log(n + 1)^{-q \beta} (1 + \smallO(1))^{1/2}} \, \frac{q \beta^2 (1 + \smallO(1))}{2 \log(n + 1)^2} = \frac{q \beta^2 n^{1/2} (1 + \smallO(1))}{2 \log(n + 1)^2},
\end{equation*}
whence \(z = \infty\) as claimed. For \(\beta = 0\) we have (again with some \(c \in \R \setminus \{0\}\))
\begin{align*}
h_n &= \ee^\alpha (1 - q \alpha)^{1/q} \, \frac{n^\alpha}{\ee^{\alpha}} \Bigl( 1 + \frac{\alpha \log(n) (1 + \smallO(1))}{2 n} \Bigr) \frac{1}{n^{1/q}} \biggl( \frac{n^{1 - q \alpha}}{1 - q \alpha} \Bigl( 1 + \frac{c (1 + \smallO(1))}{n^{\min\{1, 1 - q \alpha\}}} \Bigr) \biggr)^{1/q}\\
&= 1 + \Bigl( \frac{\alpha \log(n)}{2 n} + \frac{c}{q n^{\min\{1, 1 - q \alpha\}}} \Bigr) (1 + \smallO(1)),
\end{align*}
this yields
\begin{align*}
z_n &= \frac{\frac{n^{1 - q \alpha}}{1 - q \alpha} \bigl( 1 + \frac{c (1 + \smallO(1))}{n^{\min\{1, 1 - q \alpha\}}} \bigr)}{\frac{n^{1/2 - q \alpha}}{(1 - 2 q \alpha)^{1/2}} \bigl( 1 + \frac{c' (1 + \smallO(1))}{n^{\min\{1, 1 - 2 q \alpha\}}} \bigr)} \Bigl( \frac{\alpha \log(n)}{2 n} + \frac{c}{q n^{\min\{1, 1 - q \alpha\}}} \Bigr) (1 + \smallO(1))\\
&= \frac{(1 - 2 q \alpha)^{1/2}}{1 - q \alpha} \Bigl( \frac{\alpha \log(n)}{2 n^{1/2}} + \frac{c}{q n^{\min\{1, 1 - q \alpha\} - 1/2}} \Bigr) (1 + \smallO(1)),
\end{align*}
which converges to zero irrespective of \(\alpha\) since in any case \(\min\{1, 1 - q \alpha\} - \frac{1}{2} > 0\) because of \(\alpha < \frac{1}{2 q}\). This completes the proof.
\end{proof}

\bibliographystyle{plain}
\bibliography{volume_intersection_ellipsoids_bib}

\vspace{2\baselineskip}

\noindent
\textsc{Michael Juhos:} Faculty of Computer Science and Mathematics,
University of Passau, Innstra{\ss}e 33, 94032 Passau, Germany

\noindent
\textit{E-mail:} \texttt{michael.juhos@uni-passau.de}

\vspace{0.5\baselineskip}

\noindent
\textsc{Joscha Prochno:} Faculty of Computer Science and Mathematics,
University of Passau, Innstra{\ss}e 33, 94032 Passau, Germany

\noindent
\textit{E-mail:} \texttt{joscha.prochno@uni-passau.de}

\end{document}